\documentclass[oneside,english]{amsart}
\usepackage[T1]{fontenc}
\usepackage[latin9]{inputenc}
\usepackage{geometry}
\usepackage{amstext}
\usepackage{amsthm}
\usepackage{amssymb}
\usepackage{wasysym}
\usepackage {maple}
\usepackage[svgnames]{xcolor}
\makeatletter
\numberwithin{equation}{section}
\numberwithin{figure}{section}
\theoremstyle{plain}

\theoremstyle{plain}
\theoremstyle{definition}

\newtheorem{prop}{Proposition}[section]
\newtheorem{example}{Example}[prop]
\theoremstyle{plain}
\theoremstyle{plain}
\newtheorem*{prop*}{\protect\propositionname}
\theoremstyle{remark}
\newtheorem{rem}{Remark}[section]
\theoremstyle{plain}
\newtheorem{cor}{Corollary}[prop]

\makeatother

\usepackage{babel}

\providecommand{\propositionname}{Proposition}

\providecommand{\theoremname}{Theorem}

\begin{document}
\title{A curious multisection identity by index factorization}
\author{C. Vignat
\and 
M. Milgram\\
\\ {{Department of Mathematics, Tulane University, cvignat@tulane.edu}}\\
 Consulting Physicist, Geometrics Unlimited, Ltd., mike@geometrics-unlimited.com}
\begin{abstract}
This manuscript introduces a general multisection identity expressed
equivalently in terms of infinite double products and/or infinite double
series, from which several new product or summation identities involving special functions
including Gamma, hyperbolic trigonometric, polygamma and zeta functions, are derived.
It is shown that a parameterized version of this multisection identity
exists, a specialization of which coincides with the standard multisection
identity.
\end{abstract}

\maketitle

MSC Categories:  11Y05; 11Y40; 05E99; 40C99

\section{Introduction}

In the recent article \cite{Milgram}, the second author derived the
curious identity
\begin{equation}
\prod_{j\ge1}\left(\frac{\tan\left(\frac{a}{2^{j}}\right)}{\frac{a}{2^{j}}}\right)^{2^{j-1}}=\frac{a}{\sin a}\label{eq:tan}
\end{equation}
along with several other identities of the same flavor. 
Upon closer inspection, the first author noticed that this identity is
the specialization of a more general relationship that can be stated as
the even more curious identity
\begin{equation}
\prod_{j\ge0,n\ge1}\left(\frac{a_{\left(2n-1\right).2^{j}}}{a_{\left(2n\right).2^{j}}}\right)^{2^{j}}=\prod_{m\ge1}a_{m}\label{eq:general}
\end{equation}
that holds for any sequence of entities $\left\{ a_{m}\right\} $
such that the infinite product $\prod_{m\ge1}a_{m}$ is convergent. 
In \eqref{eq:general} we have explicitly written each index $m$ as the product of the two components of its factored form. That is, $m=(2n).2^j$ or $m=(2n-1).2^j$
according to whether $m$ is respectively even or odd (see below). 
The identity (\ref{eq:tan}) is the specialization $a_{m}=\left(1-\frac{a^{2}}{m^{2}\pi^{2}}\right)^{-1}$
of (\ref{eq:general}) ( see \eqref{TwoJ} and \eqref{Ls2} below). Identity (\ref{eq:general}) can be viewed as a \textit{ structural
identity}: it holds as a consequence of one way the
terms are grouped by the components of the factorization of their index in the product, rather than of the specific values of the individual elements $a_m$.

Other examples of structural identities are: 
\begin{itemize}
\item for an arbitrary sequence $\left\{ a_{m}\right\},$
\begin{equation}
\prod_{m\ge1}a_{2m-1}a_{2m}=\prod_{m\ge1}a_{m},
\end{equation}
which is a simple application of the dissection principle;
\item for an arbitrary summable double sequence $\left\{ a_{m,n}\right\}_{\left(m,n\right)\in\mathbb{Z}^{2}} $,
\begin{equation}
\sum_{\left(m,n\right)\in\mathbb{Z}^{2}}a_{m,n}=\sum_{m<n}a_{m,n}+\sum_{m>n}a_{m,n}+\sum_{m\in\mathbb{Z}}a_{m,m}
\end{equation}
corresponding to an obvious multisection of the two-dimensional lattice. For similar additive rather than multiplicative multisections, see \cite{beauregard_dobrushkin_2016} and \cite{Somos}.
\end{itemize}
In the following, each extension or specialization of (\ref{eq:general}) will be provided along
with its double series equivalent expression; in the case of (\ref{eq:general}),
this is
\begin{equation}
\sum_{j\ge0,n\ge1}2^{j}b_{\left(2n-1\right).2^{j}}-2^{j}b_{\left(2n\right).2^{j}}=\sum_{m\ge1}b_{m},
\end{equation}
where $b_{m}$ is an arbitrary summable sequence.

Let us now introduce some further notation:
\begin{itemize}
    \item the $b-$valuation $\nu_{b}\left(m\right)$ of a positive integer
$m$ is defined as the integer 
\begin{equation}
\label{valuation}
\nu_{b}\left(m\right)=\max\left\{ k\in\mathbb{N}:m=0\mod b^{k}\right\}.
\end{equation}
representing the largest exponent in the factorization of the integer $m$ relative to an integer $b$, (not to be confused with any of the element(s) $b_{m}$,) which will be referred to as the \textit{base}.
For example:
\begin{itemize}
    \item {$\nu_2(3)=0,$ because $3=3.2^{0}$ is not an integral multiple of $2$;}
    \item{$\nu_3(6)=1$ because $1$ is the largest exponent of $3$ in the factorization $6=2.3^{1}$;}
    \item{ $\nu_5(100)=2$ because of the decomposition $100=4.5^2$ relative to base $5$}
    \item{ and $\nu_2(100)=2$ because of the decomposition $100=25.2^2$ relative to base 2.}
\end{itemize}
 
Notice that as a consequence of this definition, for any positive integer $m,$ there is a unique representation
\begin{equation}
\label{representation}
m=n.b^{\nu_{b}\left(m\right)},\,\,b \nmid n. 
\end{equation}

\item
For a multiset $S$ that contains the integer $m$ repeated $j$ times, we use the notation
\begin{equation}
S= \{\dots,m^{(j)},\dots \}
\end{equation}
\end{itemize}
All identities in this article are based on the following principle:
define the multisets $C_b$ and $D_b$
as 
\begin{equation}
C_b=\cup_{0< k < b}\{\left(bn-k\right).b^j, n\ge1,j\ge0\}
\end{equation}  
and 
\begin{equation}
D_b=\{\left(bn\right).b^j, n\ge1,j\ge0\}.
\end{equation}
For example, in the case $b=2$, $C_{2}$ collects all integers, and 
\begin{equation}
D_{2}=
\{ 2,4^{(2)},6,8^{(3)},10,12^{(2)},14,16^{(4)},\dots \}.
\end{equation}
In the case $b=3$, $C_3$ collects all integers
and
\begin{equation}
D_3=\{3,6,9^{(2)},12,15,18^{(2)},21,24,27^{(3)},\dots\}.
\end{equation}
Defining the multiset
\begin{equation}
E_b=\{m^{(\nu_b\left(m\right))},m\in \mathbb{N}\},
\end{equation}
we compute
\begin{equation}
E_2=\{m^{(\nu_2\left(m\right))},m\in \mathbb{N}\} = \{ 2,4^{(2)},6,8^{(3)},10,12^{(2)},14,16^{(4)},\dots\}
\end{equation}
and
\begin{equation}
E_3=\{m^{(\nu_3\left(m\right))},m\in \mathbb{N}\} = \{3,6,9^{(2)},12,15,18^{(2)},21,24,27^{(3)},\dots\},
\end{equation}
and it appears that $D_2 =E_2$ and $D_3=E_3$ while\footnote{Notice that the equality $C_b=\mathbb{N}$ is a simple consequence of the fact that each set $\{\left(bn-k\right).b^j, n\ge1,j\ge0\}$
is the set of integers having, in their base $b$ representation, $j$ trailing zeros and their $j+1$st digit equal to $b-k$. Hence these sets form a partition of $\mathbb{N}.$} $C_2=C_3=\mathbb{N}.$

This result is in fact true for any base $b$:
as will be shown in Section \ref{sec:A-general-case}, for an arbitrary base $b\ge2,$
\begin{equation}
C_b=\mathbb{N} \text{ and } 
D_b = E_b.\end{equation}
More precisely, any integer $m\ge1$ appears once  in $C_b$ (i.e. $j=\nu_b\left(m\right)$) and $\nu_b\left(m\right)$ times \footnote{if $\nu_b\left(m\right)=0$, this means that $m$ does not appear in the multiset $D_{b}$}  in $D_b$ (once for each value of $j$ such that $0 \le j \le \nu_{b}\left(m\right)-1$).  

As a consequence, for two arbitrary functions $\varphi$ and $\chi$ such that the following (creatively chosen) infinite products exist,  we have the general formula (see \eqref{main_multiplicative} below)
\begin{equation}
\label{product_intro}
\prod_{j\ge0,n\ge1}\left(\prod_{k=1}^{b-1}a_{\left(nb-k\right).b^{j}}^{\varphi\left(j\right)}\right)a_{\left(nb\right).b^{j}}^{\chi\left(j\right)}=\prod_{m\ge1}a_{m}^{\varphi\left(\nu_{b}\left(m\right)\right)+\sum_{k=0}^{\nu_{b}\left(m\right)-1}\chi\left(k\right)}
\end{equation}
whose validity follows from the equality of the collected exponent of each element $a_{m}$ on both sides of the identity. 
Since this identity is structural, i.e. a consequence of 
the fact that the two multisets of indices $C_b \cup D_b$ and $\mathbb{N}\cup E_b$ coincide, it also translates into a sum form\footnote{ this remark allows us to obtain a sum form without assuming that $b_m= \log(a_m)$}: for an arbitrary sequence $b_m$ such that the following sums exist, it is given by (see \eqref{main_additive} below)
\begin{equation}
\label{sum_intro}
\sum_{j\ge0,n\ge1}\left(\left(\sum_{k=1}^{b-1}\varphi\left(j\right)b_{\left(nb-k\right).b^{j}}\right)+\chi\left(j\right)b_{\left(nb\right).b^{j}}\right)=\sum_{m\ge1}\left(\varphi\left(\nu_{b}\left(m\right)\right)+\sum_{k=0}^{\nu_{b}\left(m\right)-1}\chi\left(k\right)\right)b_{m}.
\end{equation}

Different choices of the functions $\varphi$ and $\chi,$ for example such that 
\begin{equation}
\varphi\left(\nu_{b}\left(m\right)\right)+\sum_{k=0}^{\nu_{b}\left(m\right)-1}\chi\left(k\right)=1
\end{equation}
produce the identities that will be studied in this article.

We close this introduction by noting that: (i) the product form
\begin{equation}
\prod_{n \in C_b \cup D_b} a_n = \prod_{n \in \mathbb{N} \cup E_b} a_n
\end{equation}
and the sum form
\begin{equation}
\sum_{n \in C_b \cup D_b} b_n = \sum_{n \in \mathbb{N} \cup E_b} b_n
\end{equation} 
  in formulas \eqref{product_intro} and \eqref{sum_intro} respectively can be replaced by any symmetric form, such as for example
\begin{equation}
\sum_{\underset{n_1,n_2\in C_b \cup D_b}{n_1<n_2}}b_{n_1}b_{n_2}
= \sum_{\underset{n_1,n_2\in \mathbb{N} \cup E_b}{n_1<n_2}}b_{n_1}b_{n_2}
\end{equation} 
and (ii) at a fundamental level we are effectively introducing a form of multiplicative telescoping {\footnote{cancellation between different elements in a product}} similar to that noted in \cite[Eq. (2.40)]{Milgram}.


\section{A first generalization}
Playing with (\ref{eq:general}) suggested the  following generalization.
\begin{prop}
For an arbitrary value $q\in\mathbb{C}$ and for an arbitrary sequence
$\left\{ a_{m}\right\} $ such that $\prod_{m\ge1}a_{m}$ exists,
we have
\begin{equation}
\prod_{j\ge0,n\ge1}\left(\frac{a_{\left(2n-1\right).2^{j}}}{a_{\left(2n\right).2^{j}}^{q-1}}\right)^{q^{j}}=\prod_{m\ge1}a_{m}\label{eq:q-case}
\end{equation}
and its series version:  for an arbitrary sequence
$\left\{ b_{m}\right\} $ such that $\sum_{m\ge1}b_{m}$ exists,
\begin{equation}
\sum_{j\ge0,n\ge1}q^{j}b_{\left(2n-1\right).2^{j}}-q^{j}\left(q-1\right)b_{\left(2n\right).2^{j}}=\sum_{m\ge1}b_{m}.
\label{q-caseLn}
\end{equation}
\end{prop}
A proof of this identity  
is provided
in Section \ref{sec:A-general-case} below.
A proof in the particular case $b_{m}=t^{m}$ can be found in
Section \ref{sec:A-generating-functional} below. 
Specializations of identity (\ref{eq:q-case}) are given next.
\begin{example}
For $q=0,$ identity (\ref{eq:q-case}) reduces to the usual dissection
identity
\begin{equation}
\prod_{n\ge1}a_{2n-1}a_{2n}=\prod_{n\ge1}a_{n}.
\end{equation}
\end{example}

\begin{proof}
Notice that 
\begin{equation}
q^{j}=\begin{cases}
1 & \text{if }j=0\\
0 & \text{else}
\end{cases}
\end{equation}
so that the right-hand side reduces to 
\begin{equation}
\prod_{j\ge0,n\ge1}\left(\frac{a_{\left(2n-1\right).2^{j}}}{a_{\left(2n\right).2^{j}}^{q-1}}\right)^{q^{j}}=\prod_{n\ge1}\frac{a_{\left(2n-1\right)}}{a_{2n}^{-1}}=\prod_{n\ge1}a_{2n-1}a_{2n}=\prod_{n\ge1}a_{n}.
\end{equation}
\end{proof}
\begin{rem}
The fact that we recover the usual dissection identity for $q=0$
shows that identity (\ref{eq:q-case}) can be considered as a parameterized
version of the usual dissection formula.
\end{rem}

\begin{example} \label{eq:X4}
The case $q=1$ of identity (\ref{eq:q-case}) produces
\begin{equation}
\prod_{j\ge1,n\ge1}a_{\left(2n-1\right).2^{j}}=\prod_{n\ge1}a_{2n}.
\end{equation}
\end{example}

\begin{proof}
Start from
\begin{align*}
\prod_{m\ge1}a_{m}=\prod_{j\ge0,n\ge1}\left(\frac{a_{\left(2n-1\right).2^{j}}}{a_{2n.2^{j}}^{q-1}}\right)^{q^{j}} & =\prod_{j\ge0,n\ge1}a_{\left(2n-1\right).2^{j}}\\
 & =\prod_{n\ge1}a_{2n-1}\prod_{j\ge1,n\ge1}a_{\left(2n-1\right).2^{j}}
\end{align*}
from which we deduce
\begin{equation}
\prod_{j\ge1,n\ge1}a_{\left(2n-1\right).2^{j}}=\prod_{n\ge1}a_{2n}.
\end{equation}
\end{proof}
This result is interpreted as follows: consider an arbitrary even
number $m$. By \eqref{representation},
there is a unique way to write 
\begin{equation}
m=p.2^{\nu_{2}\left(m\right)}
\end{equation}
with $p$ an odd number: first compute its valuation $\nu_{2}\left(m\right)$ according to \eqref{valuation} and
then consider
\begin{equation}
p=\frac{m}{2^{\nu_{2}\left(m\right)}}
\end{equation}
which by definition is an odd number. 
\begin{example}
In the case $q=-1$, (\ref{eq:q-case}) produces
\begin{equation}
\prod_{j\ge0,n\ge1}\frac{a_{\left(2n-1\right).2^{2j}}a_{2n.2^{2j}}^{2}}{a_{\left(2n-1\right).2^{2j+1}}a_{2n.2^{2j+1}}^{2}}=\prod_{m\ge1}a_{m}.
\end{equation}
\end{example}

\begin{proof}
For $q=-1,$ we have
\begin{equation}
\prod_{j\ge0,n\ge1}\left(\frac{a_{\left(2n-1\right).2^{j}}}{a_{2n.2^{j}}^{q-1}}\right)^{q^{j}}=\prod_{j\ge0,n\ge1}\left(a_{\left(2n-1\right).2^{j}}a_{2n.2^{j}}^{2}\right)^{\left(-1\right)^{j}}.
\end{equation}
Separating the terms with even values of $j$ (with a $+1$ exponent)
from those with odd values of $j$ (with a $-1$ power) produces the
result.
\end{proof}
The results can be extended to the base 3 case as follows.
\begin{cor}
For any $q\in\mathbb{C},$
\begin{equation}
\prod_{n\ge1,j\ge0}\left(\frac{a_{\left(3n-1\right).3^{j}}a_{\left(3n-2\right).3^{j}}}{a_{3n.3^{j}}^{q-1}}\right)^{q^{j}}=\prod_{m\ge1}a_{m}
\end{equation}
The case $q=1$ produces 
\begin{equation}
\prod_{n\ge1,j\ge0}a_{\left(3n-1\right).3^{j}}a_{\left(3n-2\right).3^{j}}=\prod_{m\ge1}a_{m}\label{eq:linear}
\end{equation}
from which we deduce the identity
\begin{equation}
\prod_{n\ge1,j\ge1}a_{\left(3n-1\right).3^{j}}a_{\left(3n-2\right).3^{j}}=\prod_{m\ge1}a_{3m}.
\label{X5b3}
\end{equation}
The case $q=0$ produces, as previously, the usual multisection
identity 
\begin{equation}
\prod_{n\ge1}a_{3n}a_{3n-1}a_{3n-2}=\prod_{m\ge1}a_{m}.
\end{equation}
\end{cor}
The arbitrary base $b$ case follows.
\begin{prop}
For any integer $n\ge 2$ and any $q\in\mathbb{C},$
\begin{equation}
\prod_{n\ge1,j\ge0}\left(\frac{a_{\left(bn-1\right).b^{j}}a_{\left(bn-2\right).b^{j}}\dots a_{\left(bn-\left(b-1\right)\right).b^{j}}}{a_{bn.b^{j}}^{q-1}}\right)^{q^{j}}=\prod_{m\ge1}a_{m}
\end{equation}
The case $q=1$ produces 
\begin{equation}
\prod_{n\ge1,j\ge0}a_{\left(bn-1\right).b^{j}}a_{\left(bn-2\right).b^{j}}\dots a_{\left(bn-\left(b-1\right)\right).b^{j}}=\prod_{m\ge1}a_{m}\label{eq:linear-1}
\end{equation}
from which we deduce the identity
\begin{equation}
\prod_{n\ge1,j\ge1}a_{\left(bn-1\right).b^{j}}a_{\left(bn-2\right).b^{j}}\dots a_{\left(bn-\left(b-1\right)\right).b^{j}}=\prod_{m\ge1}a_{bm}.
\end{equation}
The case $q=0$ produces, as previously, the usual multisection
identity 
\begin{equation}
\prod_{n\ge1}a_{\left(bn\right).n^{j}}a_{\left(bn-1\right).b^{j}}a_{\left(bn-2\right).b^{j}}\dots a_{\left(bn-\left(b-1\right)\right).b^{j}}=\prod_{m\ge1}a_{m}.
\end{equation}
\end{prop}

\section{A general case\label{sec:A-general-case}}

A general version of identity (\ref{eq:general}) is given next.
\begin{prop}
For two functions $\varphi$ and $\chi$, and with $\nu_{2}\left(m\right)$
being the $2-$valuation of $m,$ assuming that the infinite products are
convergent, then
\begin{equation}
\prod_{j\ge0,n\ge1}a_{\left(2n-1\right).2^{j}}^{\varphi\left(j\right)}a_{\left(2n\right).2^{j}}^{\chi\left(j\right)}=\prod_{m\ge1}a_{m}^{\varphi\left(\nu_{2}\left(m\right)\right)+\sum_{k=0}^{\nu_{2}\left(m\right)-1}\chi\left(k\right)}
\end{equation}
or equivalently
\begin{equation}
\sum_{j\ge0,n\ge1}\varphi\left(j\right)b_{\left(2n-1\right).2^{j}}+\chi\left(j\right)b_{\left(2n\right).2^{j}}=\sum_{m\ge1}\left(\varphi\left(\nu_{2}\left(m\right)\right)+\chi\left(0\right)+\dots+\chi\left(\nu_{2}\left(m\right)-1\right)\right)b_{m}.
\end{equation}

The base 3 case is 
\begin{equation}
\prod_{j\ge0,n\ge1}a_{\left(3n-2\right).3^{j}}^{\varphi\left(j\right)}a_{\left(3n-1\right).3^{j}}^{\varphi\left(j\right)}a_{\left(3n\right).3^{j}}^{\chi\left(j\right)}=\prod_{m\ge1}a_{m}^{\varphi\left(\nu_{3}\left(m\right)\right)+\sum_{k=0}^{\nu_{3}\left(m\right)-1}\chi\left(k\right)}
\end{equation}
or
\begin{equation}
\sum_{j\ge0,n\ge1}\varphi\left(j\right)b_{\left(3n-2\right).3^{j}}+\varphi\left(j\right)b_{\left(3n-1\right).3^{j}}+\chi\left(j\right)b_{\left(3n\right).3^{j}}=\sum_{m\ge1}\left(\varphi\left(\nu_{3}\left(m\right)\right)+\chi\left(0\right)+\dots+\chi\left(\nu_{3}\left(m\right)-1\right)\right)b_{m}.
\end{equation}
and the arbitrary base $b$ case, with $b\ge2,$ is
\begin{equation}
\label{main_multiplicative}
\prod_{j\ge0,n\ge1}\left(\prod_{k=1}^{b-1}a_{\left(nb-k\right).b^{j}}^{\varphi\left(j\right)}\right)a_{\left(nb\right).b^{j}}^{\chi\left(j\right)}=\prod_{m\ge1}a_{m}^{\varphi\left(\nu_{b}\left(m\right)\right)+\sum_{k=0}^{\nu_{b}\left(m\right)-1}\chi\left(k\right)}
\end{equation}
or equivalently
\begin{equation}
\label{main_additive}
\sum_{j\ge0,n\ge1}\left(\left(\sum_{k=1}^{b-1}\varphi\left(j\right)b_{\left(nb-k\right).b^{j}}\right)+\chi\left(j\right)b_{\left(nb\right).b^{j}}\right)=\sum_{m\ge1}\left(\varphi\left(\nu_{b}\left(m\right)\right)+\sum_{k=0}^{\nu_{b}\left(m\right)-1}\chi\left(k\right)\right)b_{m}.
\end{equation}
\end{prop}

\begin{proof}
The term $a_{m}$ appears once in the sequence 
\begin{equation}
\left\{ a_{\left(2n-1\right).2^{j}}\right\} _{j\ge0,n\ge1}
\end{equation}
for $j=\nu_{2}\left(m\right)$ and appears $\nu_{2}\left(m\right)$
times in the sequence
\begin{equation}
\left\{ a_{\left(2n\right).2^{j}}\right\} _{j\ge0,n\ge1}
\end{equation}
for $j=0,1,\dots,\nu_{2}\left(m\right)-1$ successively and therefore appears with equal cumulative exponent (or coefficient) on both sides of \eqref{main_multiplicative} or \eqref{main_additive}.  More generally, the term $a_{m}$ appears once in the sequence 
\begin{equation}
\left\{ a_{\left(bn-k\right).b^{j}}\right\} _{j\ge0,n\ge1}
\end{equation}
for $j=\nu_{b}\left(m\right)$ and appears $\nu_{b}\left(m\right)$
times in the sequence 
\begin{equation}
\left\{ a_{\left(bn\right).b^{j}}\right\} _{j\ge0,n\ge1}
\end{equation}
for $j=0,1,\dots,\nu_{b}\left(m\right)-1$ successively.
\end{proof}
\begin{example}
Here are a few specializations of the previous formula
\end{example}
\begin{enumerate}
\item In the case $\varphi\left(j\right)=\chi\left(j\right)=j$,
\begin{equation}
\prod_{j\ge0,n\ge1}\left(a_{\left(2n-1\right).2^{j}}a_{\left(2n\right).2^{j}}\right)^{j}=\prod_{m\ge1}a_{m}^{\frac{\nu_{2}\left(m\right)\left(\nu_{2}\left(m\right)+1\right)}{2}}
\end{equation}
or equivalently
\begin{equation}
\sum_{j\ge0,n\ge1}j\left(b_{\left(2n-1\right).2^{j}}+b_{\left(2n\right).2^{j}}\right)=\sum_{m\ge1}\frac{\nu_{2}\left(m\right)\left(\nu_{2}\left(m\right)+1\right)}{2}b_{m}.
\checked\end{equation}
\item In the case $\varphi\left(j\right)=j,\chi\left(j\right)=2j$, 
\begin{equation}
\prod_{j\ge0,n\ge1}\left(a_{\left(2n-1\right).2^{j}}a_{\left(2n\right).2^{j}}^{2}\right)^{j}=\prod_{m\ge1}a_{m}^{\nu_{2}^{2}\left(m\right)}
\end{equation}
or equivalently
\begin{equation}
\sum_{j\ge0,n\ge1}j\left(b_{\left(2n-1\right).2^{j}}+2b_{\left(2n\right).2^{j}}\right)=\sum_{m\ge1}\nu_{2}^{2}\left(m\right)b_{m}.\checked
\end{equation}
\item In the case $\varphi\left(j\right)=\chi\left(j\right)=1$,
\begin{equation}
\prod_{j\ge0,n\ge1}a_{\left(2n-1\right).2^{j}}a_{\left(2n\right).2^{j}}=\prod_{m\ge1}a_{m}^{\nu_{2}\left(m\right)+1}
\end{equation}
or equivalently
\begin{equation}
\sum_{j\ge0,n\ge1}b_{\left(2n-1\right).2^{j}}+b_{\left(2n\right).2^{j}}=\sum_{m\ge1}\left(\nu_{2}\left(m\right)+1\right)b_{m}.\checked
\end{equation}
\item The case $\varphi\left(j\right)=q^{j},\chi\left(j\right)=\left(1-q\right)q^{j}$
corresponds to the parameterized identity
\begin{equation}
\prod_{j\ge0,n\ge1}\left(\frac{a_{\left(2n-1\right).2^{j}}}{a_{\left(2n\right).2^{j}}^{q-1}}\right)^{q^{j}}=\prod_{m\ge1}a_{m}
\end{equation}
which is (\ref{eq:q-case}), and the case $\varphi\left(j\right)=2^{j},\chi\left(j\right)=-2^{j}$
corresponds to the original identity
\begin{equation}
\prod_{j\ge0,n\ge1}\left(\frac{a_{\left(2n-1\right).2^{j}}}{a_{\left(2n\right).2^{j}}}\right)^{2^{j}}=\prod_{m\ge1}a_{m}\checked
\label{TwoJ}   
\end{equation}
which is (\ref{eq:general}).
\item A last case is the choice $\chi\left(j\right)=j^{2p}$ for an integer
$p\ge1$ and $\varphi\left(j\right)=-\frac{B_{2p+1}\left(j\right)}{2p+1}$
with $B_{2p+1}\left(x\right)$ the Bernoulli polynomial of degree
$2p+1$, that produces the identity
\begin{equation}
\prod_{j\ge0,n\ge1}a_{\left(2n-1\right).2^{j}}^{\frac{B_{2p+1}\left(j\right)}{2p+1}}=\prod_{j\ge0,n\ge1}a_{\left(2n\right).2^{j}}^{j^{2p}}
\end{equation}
or
\begin{equation}
\sum_{j\ge0,n\ge1}\frac{B_{2p+1}\left(j\right)}{2p+1}b_{\left(2n-1\right).2^{j}}=\sum_{j\ge0,n\ge1}j^{2p}b_{\left(2n\right).2^{j}}\checked
\label{Bcase}
\end{equation}
\end{enumerate}
\begin{example}
Some more examples are: 
\end{example}

\begin{enumerate}
\item the case $\varphi\left(j\right)=-j^{2}+j+1,$ $\chi\left(j\right)=2j$
produces
\begin{equation}
\prod_{j\ge0,n\ge1}a_{\left(2n-1\right).2^{j}}^{-j^{2}+j+1}a_{\left(2n\right).2^{j}}^{2j}=\prod_{m\ge1}a_{m}
\label{Ex312}
\end{equation}
or
\begin{equation}
 \sum_{n\ge1,j\ge0}\left(-j^{2}+j+1\right)b_{\left(2n-1\right).2^{j}}+\left(2j\right)b_{\left(2n\right).2^{j}}=\sum_{m\ge1}b_{m}.\checked
\label{Ex312a}   
\end{equation}

\item the telescoping choice $\varphi\left(k\right)=\frac{1}{k+1},$ $\chi\left(k\right)=\frac{1}{\left(k+1\right)\left(k+2\right)}$
produces
\begin{equation}
\prod_{j\ge0,n\ge1}a_{\left(2n-1\right).2^{j}}^{\frac{1}{j+1}}a_{\left(2n\right).2^{j}}^{\frac{1}{\left(j+1\right)\left(j+2\right)}}=\prod_{m\ge1}a_{m}\label{eq:1/(k+1)}
\end{equation}
or
\begin{equation}
\sum_{j\ge0,n\ge1}\frac{1}{j+1}b_{\left(2n-1\right).2^{j}}+\frac{1}{\left(j+1\right)\left(j+2\right)}b_{\left(2n\right).2^{j}}=\sum_{m\ge1}b_{m}\checked
\label{eq:1/(k+1)B}
\end{equation}
\item the base 3 case of the previous identity is
\begin{equation}
\prod_{j\ge0,n\ge1}a_{\left(3n-1\right).3^{j}}^{\frac{1}{j+1}}a_{\left(3n-2\right).3^{j}}^{\frac{1}{j+1}}a_{\left(3n\right).3^{j}}^{\frac{1}{\left(j+1\right)\left(j+2\right)}}=\prod_{m\ge1}a_{m}
\end{equation}
or
\begin{equation}
\sum_{j\ge0,n\ge1}\frac{1}{j+1}b_{\left(3n-1\right).3^{j}}+\frac{1}{j+1}b_{\left(3n-2\right).3^{j}}+\frac{1}{\left(j+1\right)\left(j+2\right)}b_{\left(3n\right).3^{j}}=\sum_{m\ge1}b_{m}\checked
\end{equation}
\item the choice $\varphi\left(k\right)=1-\frac{k\left(k+3\right)}{4\left(k+1\right)\left(k+2\right)},$
$\chi\left(k\right)=\frac{1}{\left(k+1\right)\left(k+2\right)\left(k+3\right)}$
produces
\begin{equation}
\prod_{j\ge0,n\ge1}a_{\left(2n-1\right).2^{j}}^{1-\frac{j\left(j+3\right)}{4\left(j+1\right)\left(j+2\right)}}a_{\left(2n\right).2^{j}}^{\frac{1}{\left(j+1\right)\left(j+2\right)\left(j+3\right)}}=\prod_{m\ge1}a_{m}
\end{equation}
or 
\begin{equation}
\sum_{j\ge0,n\ge1}\left(1-\frac{j\left(j+3\right)}{4\left(j+1\right)\left(j+2\right)}\right)b_{\left(2n-1\right).2^{j}}+\frac{1}{\left(j+1\right)\left(j+2\right)\left(j+3\right)}b_{\left(2n\right).2^{j}}=\sum_{m\ge1}b_{m}\checked
\end{equation}
\end{enumerate}

\begin{rem}
We close this section with two specializations
\begin{itemize}
\item 
the specialization $\chi\left(k\right)=0$ produces 
generating functions for the valuation function $\nu_b(n),$ of the type
\begin{equation}
\prod_{j\ge0,n\ge1}\left(\prod_{k=1}^{b-1}a_{\left(nb-k\right).b^{j}}^{\varphi\left(j\right)}\right)=\prod_{m\ge1}a_{m}^{\varphi\left(\nu_{b}\left(m\right)\right)}
\end{equation}
or
\begin{equation}
\sum_{j\ge0,n\ge1}\varphi\left(j\right) \sum_{k=1}^{b-1}b_{\left(nb-k\right).b^{j}}=\sum_{m\ge1}\varphi\left(\nu_{b}\left(m\right)\right) b_{m};
\end{equation}
for example, in the case $b=2,$
\begin{equation}
\prod_{j\ge0,n\ge1}a_{\left(2n-1\right).2^{j}}^{\varphi\left(j\right)}=\prod_{m\ge1}a_{m}^{\varphi\left(\nu_{2}\left(m\right)\right)}
\end{equation}
or equivalently
\begin{equation}
\sum_{j\ge0,n\ge1}\varphi\left(j\right)b_{\left(2n-1\right).2^{j}}=\sum_{m\ge1}\varphi\left(\nu_{2}\left(m\right)\right)b_{m}.
\end{equation}
\begin{itemize}
\item 
the choice $b_{m}=t^{m}$ produces a generating function for
the sequence $\left\{ \varphi\left(\nu_{2}\left(m\right)\right)\right\} $
in the form
\begin{equation}
\sum_{j\ge0}\varphi\left(j\right)\frac{t^{2^{j}}}{1-t^{2^{j+1}}}=\sum_{m\ge1}\varphi\left(\nu_{2}\left(m\right)\right)t^{m}
\end{equation}
\item the choice $b_{m}=\frac{1}{m^{s}}$  produces the identity between Dirichlet
series
\begin{equation}
\left(1-2^{-s}\right)\zeta\left(s\right)\sum_{j\ge0}\frac{\varphi\left(j\right)}{2^{s j}}=\sum_{m\ge1}\frac{\varphi\left(\nu_{2}\left(m\right)\right)}{m^{s}};
\end{equation}
choosing $\varphi\left(j\right) = \cos\left(2 \pi j x\right)$ 
produces
\begin{align}
\label{cos}
\sum_{m\ge1} \frac{\cos \left(2\pi x \nu_2\left(m\right)\right)}{m^s}
&=
\frac{2^s-1}{2^{s+1}}
\frac{\cos\left(2\pi x\right)-2^s}{\cos\left(2\pi x\right)-\frac{2^s+2^{-s}}{2}}
\zeta\left(s \right).
\end{align}
The specializations $x=\frac{1}{4}$ and $x=\frac{1}{2}$ successively produce, for $s >1,$ the  identities
\begin{equation}
\sum_{m\ge1} \frac{\cos \left(\frac{\pi}{2} \nu_2\left(m\right)\right)}{m^{s}}=
\frac{4^s-2^s}{4^s+1}\zeta\left(s\right)
\end{equation}
and
\begin{equation}
\label{x=1/2}
\sum_{m\ge1} \frac{\left(-1\right)^{\nu_2\left(m\right)}}{m^{s}}=
\frac{2^s-1}{2^s+1}\zeta\left(s\right).
\end{equation}
This identity should be compared to the Dirichlet series (see OEIS entry A007814)
$$
\sum_{m\ge1}
\frac{\nu_2(m) }{m^s}=\frac{\zeta(s)}{2^s-1}
$$
\end{itemize}
\item
the specialization $\varphi\left(k\right)=0$ produces
identities of the type
\begin{equation}
\prod_{j\ge0,n\ge1}a_{\left(nb\right).b^{j}}^{\chi\left(j\right)}=\prod_{m\ge1}a_{m}^{\sum_{k=0}^{\nu_{b}\left(m\right)-1}\chi\left(k\right)}
\end{equation}
or
\begin{equation}
\sum_{j\ge0,n\ge1}\chi\left(j\right) b_{\left(bn\right).b^{j}}
=\sum_{m\ge1}\left(
\sum_{k=0}^{\nu_{b}\left(m\right)-1}\chi\left(k\right)
\right)b_{m};
\end{equation}
for example, if $b=2$,
\begin{equation}
\prod_{j\ge0,n\ge1}a_{\left(2n\right).2^{j}}^{\chi\left(j\right)}=\prod_{m\ge1}a_{m}^{\sum_{k=0}^{\nu_{2}\left(m\right)-1}\chi\left(k\right)}
\end{equation}
or equivalently
\begin{equation}
\sum_{j\ge0,n\ge1}\chi\left(j\right) b_{\left(2n\right).2^{j}}
=\sum_{m\ge1}\left(
\sum_{k=0}^{\nu_{2}\left(m\right)-1}\chi\left(k\right)
\right)b_{m}.
\end{equation}
The specialization $b_m=\frac{1}{m^{s}}$ and $\chi\left(j\right)=t^{j}$ yields
the Dirichlet series
\begin{equation}
\sum_{m \ge 1} 
\frac{t^{\nu_2\left(m\right)}}{m^s}
=\left(\frac{2^s -1}{2^s - t}\right)
\zeta\left(s \right),
\end{equation}
the special case $t=-1$ of which reduces to \eqref{x=1/2} while the case $t=e^{\imath 2 \pi x}$ reduces to \eqref{cos}.
\end{itemize}
\end{rem}
\begin{rem}
Identity \eqref{x=1/2} is in fact easy to recover directly from the observation that $\nu_2(m)=0$ for $m$ odd while $\nu_2(m)$ is even or odd according to $j=0$ or not in the factorization $m=(2m-1).2^j.$ Employing this fact allows the odd terms of the sum in \eqref{cos} to be evaluated, leading to
\begin{equation}
    \moverset{\infty}{\munderset{m =1}{\sum}}\! \frac{\cos \! \left(2\,\pi \,x\,\nu_{2} \! \left(2\,m \right)\right)}{\left(2\,m \right)^{s}}
 = 
-\frac{\left(\cos \! \left(2\,\pi \,x \right) \left(2^{s}-1\right)-1+2^{-s}\right) \zeta \! \left(s \right)}{2\,\cos \! \left(2\,\pi \,x \right) 2^{s}-4^{s}-1};
\label{Eq3p5b}
\end{equation}
Finally, take the limit $s\rightarrow 1$, with $x=1/6$ when the series converges, to find
\begin{equation}
\moverset{\infty}{\munderset{m =1}{\sum}}\! \frac{\cos \! \left({\pi \,\nu_{2} \left(2\,m \right)}/{3}\right)}{m }
 = \frac{\ln \! \left(2\right)}{3}
\,.
\label{Eq3p5b1}
\end{equation}
\end{rem}
\section{a finite version}
A sum or product over a finite range for the index $j$ is stated next.
\begin{prop}
For any integer $J\ge1$, the following identity holds
\begin{equation}
\prod_{j\ge J,n\ge1}a_{\left(2n-1\right).2^{j}}^{\varphi\left(j\right)}=\prod_{p\ge1}a_{p.2^{J}}^{\varphi\left(J+s_{2}\left(p\right)\right)}.
\end{equation}
\end{prop}
\begin{proof}
Choosing an integer $J\ge1$ and replacing the function $\varphi\left(j\right)$
with its truncated version
\begin{equation}
\varphi_{J}\left(j\right)=\begin{cases}
\varphi\left(j\right) & j\ge J\\
0 & \text{else}
\end{cases}
\end{equation}
in 
\begin{equation}
\prod_{j\ge1,n\ge1}a_{\left(2n-1\right).2^{j}}^{\varphi\left(j\right)}=\prod_{m\ge1}a_{m}^{\varphi\left(s_{2}\left(m\right)\right)}.
\end{equation}
produces the identity
\begin{equation}
\prod_{j\ge J,n\ge1}a_{\left(2n-1\right).2^{j}}^{\varphi\left(j\right)}=\prod_{p\ge1}a_{p.2^{J}}^{\varphi\left(s_{2}\left(p.2^{J}\right)\right)}.
\end{equation}
Indeed, using
\begin{equation}
\prod_{j\ge0,n\ge1}a_{\left(2n-1\right).2^{j}}^{\varphi_{J}\left(j\right)}=\prod_{m\ge1}a_{m}^{\varphi_{J}\left(s_{2}\left(m\right)\right)}
\end{equation}
we look for the values of the index $m$ such that $s_{2}\left(m\right)\ge J,$
which are exactly 
\begin{equation}
m=2^{J}p,\,\,p\in\mathbb{N}
\end{equation}
so that
\begin{equation}
\prod_{j\ge J,n\ge1}a_{\left(2n-1\right).2^{j}}^{\varphi\left(j\right)}=\prod_{p\ge1}a_{p.2^{J}}^{\varphi\left(s_{2}\left(p.2^{J}\right)\right)}.
\end{equation}
Since moreover $s_{2}\left(p.2^{J}\right)=J+s_{2}\left(p\right),$
we deduce the result.
\end{proof}
\begin{cor}
The specialization $\varphi\left(j\right)=1$ produces
\begin{equation}
\prod_{j\ge J,n\ge1}a_{\left(2n-1\right).2^{j}}=\prod_{p\ge1}a_{p.2^{J}}
\end{equation}
or equivalently
\begin{equation}
\prod_{0\le j<J,n\ge1}a_{\left(2n-1\right).2^{j}}
=\prod_{p\ge1,p\not\equiv0\mod2^{J}}a_{p}
\end{equation}

or its sum version
\begin{equation}
\sum_{j\ge J,n\ge1}b_{\left(2n-1\right).2^{j}}=\sum_{p\ge1,p\not\equiv0\mod2^{J}}b_{p}
\end{equation}
\end{cor}
\section{The case of a double-indexed sequence}
\begin{prop}
Consider a double-indexed sequence $\left\{ a_{p,q}\right\} $. We
have
\begin{equation}
\label{double product version}    
\prod_{k\ge0,m\ge1}\prod_{j\ge0,n\ge1}\left(\frac{a_{\left(2n-1\right).2^{j},\left(2m-1\right).2^{k}}a_{\left(2n\right).2^{j},\left(2m\right).2^{k}}}{a_{\left(2n\right).2^{j},\left(2m-1\right).2^{k}}a_{\left(2n-1\right).2^{j},\left(2m\right).2^{k}}}\right)^{2^{j+k}}=\prod_{p\ge1,q\ge1}a_{p,q}.
\end{equation}
The sum version is, for an arbitrary sequence $\left\{ b_{p,q}\right\} ,$
\begin{align}
\label{double sum version}
\sum_{p,q\ge1}b_{p,q} & =\sum_{j,k\ge0}\sum_{m,n\ge1}2^{j+k}\left[b_{\left(2n-1\right).2^{j},\left(2m-1\right).2^{k}}+b_{\left(2n\right).2^{j},\left(2m\right).2^{k}}\right.\\
\nonumber
 & -\left.b_{\left(2n\right).2^{j},\left(2m-1\right).2^{k}}-b_{\left(2n-1\right).2^{j},\left(2m\right).2^{k}}\right]
\end{align}
$ $
\end{prop}
\begin{proof}
For a fixed value of $q$ we have
\begin{equation}
\prod_{j\ge0,n\ge1}\left(\frac{a_{\left(2n-1\right).2^{j},q}}{a_{\left(2n\right).2^{j},q}}\right)^{2^{j}}=\prod_{p\ge1}a_{p,q}.
\end{equation}
Denote 
\begin{equation}
b_{q}=\prod_{j\ge0,n\ge1}\left(\frac{a_{\left(2n-1\right).2^{j},q}}{a_{\left(2n\right).2^{j},q}}\right)^{2^{j}}
\end{equation}
so that
\begin{equation}
\prod_{q\ge1}b_{q}=\prod_{p\ge1,q\ge1}a_{p,q}.
\end{equation}
Using
\begin{equation}
\prod_{q\ge1}b_{q}=\prod_{k\ge0,m\ge1}\left(\frac{b_{\left(2m-1\right).2^{k}}}{b_{\left(2m\right).2^{k}}}\right)^{2^{k}}
\end{equation}
we deduce 
\begin{align*}
\prod_{p\ge1}a_{p,q} & =\prod_{k\ge0,m\ge1}\left(\frac{\prod_{j\ge0,n\ge1}\left(\frac{a_{\left(2n-1\right).2^{j},\left(2m-1\right).2^{k}}}{a_{\left(2n\right).2^{j},\left(2m-1\right).2^{k}}}\right)^{2^{j}}}{\prod_{j\ge0,n\ge1}\left(\frac{a_{\left(2n-1\right).2^{j},\left(2m\right).2^{k}}}{a_{\left(2n\right).2^{j},\left(2m\right).2^{k}}}\right)^{2^{j}}}\right)^{2^{k}}\\
 & =\prod_{k\ge0,m\ge1}\prod_{j\ge0,n\ge1}\left(\frac{a_{\left(2n-1\right).2^{j},\left(2m-1\right).2^{k}}a_{\left(2n\right).2^{j},\left(2m\right).2^{k}}}{a_{\left(2n\right).2^{j},\left(2m-1\right).2^{k}}a_{\left(2n-1\right).2^{j},\left(2m\right).2^{k}}}\right)^{2^{j+k}}
\end{align*}
\end{proof}
\begin{cor}
The two Lambert series
\begin{equation}
f\left(q\right)=\sum_{n\ge1}n\frac{q^{n}}{1-q^{n}}\text{ and }\,g\left(q\right)=\sum_{n\ge1}n\frac{q^{n}}{1+q^{n}}
\end{equation}
are related as
\begin{equation}
f\left(q\right)=\sum_{j,k\ge0}2^{2j+k}\left[g\left(q^{2^{j+k}}\right)-4g\left(q^{2^{j+k+1}}\right)\right].
\end{equation}    
\end{cor}
\begin{proof}
Writing the function $f$ as the double sum
\begin{equation}
f\left(q\right)=\sum_{m,n}nq^{nm},
\end{equation}
the dissection formula \eqref{double sum version} yields
\begin{align*}
f\left(q\right) & =\sum_{j,k\ge0}\sum_{n,m\ge1}2^{j+k}\left[\left(2n-1\right)2^{j}q^{\left(2n-1\right).2^{j}.\left(2m-1\right).2^{k}}-\left(2n-1\right)2^{j}q^{\left(2n-1\right).2^{j}.\left(2m\right).2^{k}}\right.\\
 & +\left.\left(2n\right).2^{j}q^{\left(2n\right).2^{j}.\left(2m\right).2^{k}}-\left(2n\right)2^{j}q^{\left(2n\right).2^{j}.\left(2m-1\right).2^{k}}\right].
\end{align*}
Each of the 4 sums over $n,m\ge1$ is computed as follows
\begin{equation}
\sum_{n,m\ge1}\left(2n-1\right)q^{\left(2n-1\right).2^{j}.\left(2m-1\right).2^{k}}=\sum_{n\ge1}\left(2n-1\right)\frac{q^{\left(2n-1\right)2^{j+k}}}{1-q^{\left(2n-1\right)2^{j+k+1}}},
\end{equation}
\begin{equation}
\sum_{n,m\ge1}\left(2n-1\right)q^{\left(2n-1\right).2^{j}.\left(2m\right).2^{k}}=\sum_{n\ge1}\left(2n-1\right)\frac{q^{\left(2n-1\right)2^{j+k+1}}}{1-q^{\left(2n-1\right)2^{j+k+1}}},
\end{equation}
\begin{equation}
\sum_{n,m\ge1}\left(2n\right)q^{\left(2n\right).2^{j}.\left(2m\right).2^{k}}=\sum_{n\ge1}2n\frac{q^{n.2^{j+k+2}}}{1-q^{n.2^{j+k+2}}},
\end{equation}
\begin{equation}
\sum_{n,m\ge1}\left(2n\right)q^{\left(2n\right).2^{j}.\left(2m-1\right).2^{k}}=\sum_{n\ge1}\left(2n\right)\frac{q^{n.2^{j+k+1}}}{1-q^{n.2^{j+k+2}}}.
\end{equation}
With the notation $q_{j,k}=q^{2^{j+k}},$ this yields
\begin{align*}
f\left(q\right) & =\sum_{j,k\ge0}2^{2j+k}\sum_{n\ge1}\left[\left(2n-1\right)\frac{q_{j,k}^{2n-1}-q_{j,k}^{2\left(2n-1\right)}}{1-q_{j,k}^{2\left(2n-1\right)}}+\left(2n\right)\frac{q_{j,k}^{4n}-q_{j,k}^{2n}}{1-q_{j,k}^{4n}}\right]\\
 & =\sum_{j,k\ge0}2^{2j+k}\sum_{n\ge1}\left[\left(2n-1\right)\frac{q_{j,k}^{2n-1}}{1+q_{j,k}^{2n-1}}-\left(2n\right)\frac{q_{j,k}^{2n}}{1+q_{j,k}^{2n}}\right].
\end{align*}
This inner sum coincides with the difference between the odd part
$h_{o}$ and the even part $h_{e}$ of the function $h\left(q\right)=\sum_{n\ge1}n\frac{q_{j,k}^{n}}{1+q_{j,k}^{n}},$
also equal to $h\left(q\right)-2h_{e}\left(q\right),$ i.e. to
\begin{equation}
\sum_{n\ge1}\left[n\frac{q_{j,k}^{n}}{1+q_{j,k}^{n}}-4n\frac{q_{j,k}^{2n}}{1+q_{j,k}^{2n}}\right],
\end{equation}
so that finally
\begin{align*}
f\left(q\right) & =\sum_{j,k\ge0}2^{2j+k}\sum_{n\ge1}\left[n\frac{q^{n.2^{j+k}}}{1+q^{n.2^{j+k}}}-4n\frac{q^{n.2^{j+k+1}}}{1+q^{n.2^{j+k+1}}}\right]\\
 & =\sum_{j,k\ge0}2^{2j+k}\left(g\left(q^{2^{j+k}}\right)-4g\left(q^{2^{j+k+1}}\right)\right).
\end{align*}    
\end{proof}
We notice that the previous result can be extended in a straightforward way to the following result; we skip the details.
\begin{cor}
For $\mu\ge1,$ the two Lambert series
\begin{equation}
f\left(q\right)=\sum_{n\ge1}n^\mu\frac{q^{n}}{1-q^{n}}\text{ and }\,g\left(q\right)=\sum_{n\ge1}n^\mu\frac{q^{n}}{1+q^{n}}
\end{equation}
are related as
\begin{equation}
f\left(q\right)=\sum_{j,k\ge0}2^{2j+k}\left[g\left(q^{2^{j+k}}\right)-2^{\mu+1}g\left(q^{2^{j+k+1}}\right)\right].
\end{equation}    
\end{cor}

\section{A generating functional approach\label{sec:A-generating-functional}}

\subsection{The base 2 case}
\subsubsection{Two proofs of Teixeira's identity}
The series version of (\ref{eq:general}) applied to the case $a_{m}=t^{m}$ is
\begin{equation}
\sum_{m\ge1}t^{m}=\sum_{k\ge0,n\ge1}2^{k}t^{\left(2n-1\right)2^{k}}-2^{k}t^{\left(2n\right)2^{k}}
\end{equation}
Summing over $n$ in the right hand-side and using the simplification 
\begin{equation}
\sum_{k\ge0}2^{k}\frac{t^{2^{k}}\left(1-t^{2^{k}}\right)}{1-t^{2^{k+1}}}=
\sum_{k\ge0}2^{k}\frac{t^{2^{k}}}{1+t^{2^{k}}}
\end{equation}
produces
\begin{equation}
\frac{t}{1-t}=\sum_{k\ge0}2^{k}\frac{t^{2^{k}}}{1+t^{2^{k}}}.\label{eq:Teixeira}
\end{equation}
This identity is well-known and
appears as 
Problem 10 of Chapter II in \cite[p. 118]{Comtet}, where the result
is attributed to F.G. Teixeira \cite{Teixeira}.
\subsubsection*{first proof}
A simple proof of (\ref{eq:Teixeira}) is as follows: define 
\begin{equation}
\varphi\left(t\right)=\frac{t}{1-t},
\end{equation}
notice that this function satisfies
\begin{equation}
\varphi\left(t\right)+\varphi\left(-t\right)=2\varphi\left(t^{2}\right)
\end{equation}
and iterate this identity (see \cite[Appendix A]{Milgram} for convergence details).
\subsubsection*{second proof}
Another proof appeals to the theory of Lambert series: first, recall the definition of the  Dirichlet eta function 
\begin{equation}
\eta\left(s\right)=\sum_{n\ge1}\frac{\left(-1\right)^{n-1}}{n^{s}}=\left(1-2^{1-s}\right)\zeta\left(s\right).
\label{eq:eta}
\end{equation}
Then consider the following result provided in \cite[Problem 11, p.300]{Borwein}
\begin{prop}
Given two sequences $\left\{ a_{n}\right\} $ and $\left\{ b_{n}\right\} $, 
\begin{equation}
\label{Borwein1}
\eta\left(s\right)\sum_{n\ge1}\frac{a_{n}}{n^{s}}=\sum_{n\ge1}\frac{b_{n}}{n^{s}}    
\end{equation}
if and only if
\begin{equation}
\label{Borwein2}
\sum_{n\ge1}a_{n}\frac{x^{n}}{1+x^{n}}=\sum_{n\ge1}b_{n}x^{n}.    
\end{equation}
\end{prop}
Let us apply this result to the sequence $\{a_n\}$ defined as
\begin{equation}
a_n=
\begin{cases}
2^l \text{ if } n=2^l\\
0 \text{ else}
\end{cases}
\end{equation}
to obtain, on the left-hand side of \eqref{Borwein2},
\begin{equation}
\sum_{n\ge1}a_n \frac{x^n}{1+x^n}=
\sum_{l\ge1}2^l \frac{x^{2^l}}{1+x^{2^l}}
\end{equation}
and, on the left-hand side of \eqref{Borwein1},
\begin{equation}
\eta(s)\sum_{n\ge1}\frac{a_n}{n^s}
=\eta(s)\sum_{l\ge0}\frac{1}{2^{l(s-1)}}
=\zeta(s).
\end{equation}
We deduce from \eqref{Borwein1} that
\begin{equation}
\sum_{n\ge1}\frac{b_n}{n^s} = \zeta(s)
\end{equation}
so that $b_n=1,\,\,n\ge1$ and, as a consequence, Teixeira's identity \eqref{eq:Teixeira}.

A simple interpretation of Teixeira's identity \eqref{eq:Teixeira}
is as follows:
\begin{equation}
\frac{t}{1-t}=\sum_{n\ge1}t^{n}
\end{equation}
is the generating function of the sequence $\left\{ 1,1,\dots\right\} :$
each integer $n$ is counted with a unit weight. The first term $\left(j=0\right)$
in the right-hand side of Teixeira's identity (\ref{eq:Teixeira})
is 
\begin{equation}
\frac{t}{1+t}=\sum_{n\ge1}\left(-1\right)^{n}t^{n}
\end{equation}
so that all even integers are weighted by $+1$ but the odd ones are
weighted by $\left(-1\right).$ The extra terms 
\begin{equation}
\sum_{j\ge1}2^{j}\frac{t^{2^{j}}}{1+t^{2^{j}}}
\end{equation}
will correct these negative weights so that the final sum has all positive
unit weights. More precisely:
\begin{align*}
\frac{t}{1+t} & =t-t^{2}+t^{3}-t^{4}+t^{5}-t^{6}+t^{7}-t^{8}+t^{9}-t^{10}+t^{11}-t^{12}\\
 & +t^{13}-t^{14}+t^{15}-t^{16}+t^{17}-t^{18}+t^{19}-t^{20}+O\left(t^{21}\right)
\end{align*}
\begin{align*}
\frac{t}{1+t}+2\frac{t^{2}}{1+t^{2}} & =t+t^{2}+t^{3}-3t^{4}+t^{5}+t^{6}+t^{7}-3t^{8}+t^{9}+t^{10}+t^{11}-3t^{12}\\
 & +t^{13}+t^{14}+t^{15}-3t^{16}+t^{17}+t^{18}+t^{19}-3t^{20}+O\left(t^{21}\right)
\end{align*}
\begin{align*}
\frac{t}{1+t}+2\frac{t^{2}}{1+t^{2}}+4\frac{t^{4}}{1+t^{4}} & =t+t^{2}+t^{3}+t^{4}+t^{5}+t^{6}+t^{7}-7t^{8}+t^{9}+t^{10}+t^{11}+t^{12}\\
 & +t^{13}+t^{14}+t^{15}-7t^{16}+t^{17}+t^{18}+t^{19}+t^{20}+O\left(t^{21}\right)
\end{align*}
so that all powers of $t$ in 
\begin{equation}
\sum_{j=0}^{J}2^{j}\frac{t^{2^{j}}}{1+t^{2^{j}}}
\end{equation}
have unit weight except, those $t$ with power $k2^{J+1},\,\,k\ge1$
that have weight $1-2^{J+1}.$ As $J\to\infty,$ only terms with unit
weight remain.

\subsubsection{A generalization of Teixeira's identity}
Now consider the functions
\begin{equation}
\varphi\left(z\right)=\frac{z}{1-z},\,\,\chi\left(z\right)=\frac{z-\left(q-1\right)z^{2}}{1-z^{2}};
\end{equation}
they satisfy the identity
\begin{equation}
\frac{1}{q}\left(\varphi\left(z\right)-\chi\left(z\right)\right)=\varphi\left(z^{2}\right).
\end{equation}
We deduce
\begin{equation}
\varphi\left(z\right)=\chi\left(z\right)+q\varphi\left(z^{2}\right)=\chi\left(z\right)+q\chi\left(z^{2}\right)+q^{2}\varphi\left(z^{4}\right)=\dots
\end{equation}
so that
\begin{align*}
\frac{z}{1-z} & =\sum_{k\ge0}q^{k}\chi\left(z^{2^{k}}\right)=\sum_{k\ge0}q^{k}\frac{z^{2^{k}}-\left(q-1\right)z^{2^{k+1}}}{1-z^{2^{k+1}}}\\
 & =\sum_{k\ge0,n\ge1}q^{k}\left(z^{\left(2n-1\right).2^{k}}-\left(q-1\right)z^{\left(2n\right).2^{k}}\right)
\end{align*}
which is the $a_{n}=z^{n}$ version of the generalized identity
\begin{equation}
\sum_{n\ge1}a_{n}=\sum_{k\ge0,n\ge1}q^{k}\left(a_{\left(2n-1\right).2^{k}}-\left(q-1\right)a_{\left(2n\right).2^{k}}\right).
\end{equation}



\subsection{The arbitrary base b case}

Consider the functions
\begin{equation}
\varphi\left(z\right)=\frac{z}{1-z},\,\,\chi\left(z\right)=\frac{z+z^{2}+\dots+z^{b-1}-\left(q-1\right)z^{b}}{1-z^{b}};
\end{equation}
they satisfy
\begin{equation}
\varphi\left(z\right)-\chi\left(z\right)=q\varphi\left(z^{b}\right)
\end{equation}
and we deduce by iterating
\begin{align*}
\varphi\left(z\right) & =\chi\left(z\right)+q\chi\left(z^{b}\right)+q^{2}\chi\left(z^{b^{2}}\right)+\dots\\
 & =\sum_{j\ge0}q^{j}\chi\left(z^{b^{j}}\right)
\end{align*}
so that
\begin{align*}
\frac{z}{1-z} & =\sum_{j\ge0}q^{j}\frac{z^{b^{j}}}{1-z^{b^{j+1}}}+q^{j}\frac{z^{2.b^{j}}}{1-z^{b^{j+1}}}+\dots+q^{j}\frac{z^{\left(b-1\right).b^{j}}}{1-z^{b^{j+1}}}-\left(q-1\right)q^{j}\frac{z^{b.b^{j}}}{1-z^{b^{j+1}}}\\
 & =\sum_{j\ge0,n\ge1}q^{j}z^{\left(bn-1\right).b^{j}}+q^{j}z^{\left(bn-2\right).b^{j}}+\dots+q^{j}z^{\left(bn-\left(b-1\right)\right).b^{j}}-\left(q-1\right)q^{j}z^{bn.b^{j}}.
\end{align*}

\section{A q-calculus application}

This section produces a $q-$calculus application of our main identity.
\begin{prop}
With the $q-$\textup{Pochhammer} symbol
\begin{equation}
\left(a;q\right)_{\infty}=\prod_{k\ge0}\left(1-aq^{k}\right),
\end{equation}
we have the identities, for $q\in\mathbb{C}$ such that $\vert q\vert<1,$
\begin{equation}
\prod_{j\ge0}\left(\frac{\left(q^{2^{j}};q^{2^{j+1}}\right)_{\infty}}{\left(q^{2^{j+1}};q^{2^{j+1}}\right)_{\infty}}\right)^{2^{j}}=\left(q;q\right)_{\infty},
\end{equation}
and more generally, for any complex number $p,$
\begin{equation}
\prod_{j\ge0}\left(\frac{\left(q^{2^{j}};q^{2^{j+1}}\right)_{\infty}}{\left(q^{2^{j+1}};q^{2^{j+1}}\right)_{\infty}^{p-1}}\right)^{p^{j}}=\left(q;q\right)_{\infty}
\end{equation}
and 
\begin{equation}
\prod_{j\ge0}\left(\frac{\left(q^{3^{j}};q^{3^{j+1}}\right)_{\infty}\left(q^{2.3^{j}};q^{3^{j+1}}\right)_{\infty}}{\left(q^{3^{j+1}};q^{3^{j+1}}\right)_{\infty}^{p-1}}\right)^{p^{j}}=\left(q;q\right)_{\infty}.
\end{equation}
The special case $p=1$ produces respectively
\begin{equation}
\label{QPochhammer2}    
\prod_{j\ge0}\left(q^{2^{j}};q^{2^{j+1}}\right)_{\infty}=\left(q;q\right)_{\infty}
\end{equation}
and
\begin{equation}
\prod_{j\ge0}\left(q^{3^{j}};q^{3^{j+1}}\right)_{\infty}\left(q^{2.3^{j}};q^{3^{j+1}}\right)_{\infty}=\left(q;q\right)_{\infty}.
\end{equation}
\end{prop}
\begin{proof}
In the identity
\begin{equation}
\prod_{n\ge1,j\ge0}\left(\frac{a_{\left(2n-1\right).2^{j}}}{a_{2n.2^{j}}}\right)^{2^{j}}=\prod_{n\ge1}a_{n},
\end{equation}
take $a_{n}=1-q^{n}$ so that the right-hand side is $\left(q;q\right)_{\infty}$.
Moreover,
\begin{equation}
\prod_{n\ge1,j\ge0}\left(1-q^{\left(2n-1\right).2^{j}}\right)=\prod_{n\ge0,j\ge0}\left(1-q^{\left(2n+1\right).2^{j}}\right)=\prod_{j\ge0}\left(q^{2^{j}};q^{2^{j+1}}\right)_{\infty}
\end{equation}
and
\begin{equation}
\prod_{n\ge1,j\ge0}\left(1-q^{2n.2^{j}}\right)=\prod_{n\ge0,j\ge0}\left(1-q^{\left(2n+2\right).2^{j}}\right)=\prod_{j\ge0}\left(q^{2^{j+1}};q^{2^{j+1}}\right)_{\infty}
\end{equation}
so that
\begin{equation}
\prod_{j\ge0}\left(\frac{\left(q^{2^{j}};q^{2^{j+1}}\right)_{\infty}}{\left(q^{2^{j+1}};q^{2^{j+1}}\right)_{\infty}}\right)^{2^{j}}=\left(q;q\right)_{\infty}.
\end{equation}
In the more general case
\begin{equation}
\prod_{n\ge1,j\ge0}\left(\frac{a_{\left(2n-1\right).2^{j}}}{a_{2n.2^{j}}^{p-1}}\right)^{p^{j}}=\prod_{n\ge1}a_{n},
\end{equation}
we deduce
\begin{equation}
\prod_{j\ge0}\left(\frac{\left(q^{2^{j}};q^{2^{j+1}}\right)_{\infty}}{\left(q^{2^{j+1}};q^{2^{j+1}}\right)_{\infty}^{p-1}}\right)^{p^{j}}=\left(q;q\right)_{\infty}.
\end{equation}
For example, $p=3$ produces
\begin{equation}
\prod_{j\ge0}\left(\frac{\left(q^{2^{j}};q^{2^{j+1}}\right)_{\infty}}{\left(q^{2^{j+1}};q^{2^{j+1}}\right)_{\infty}^{2}}\right)^{3^{j}}=\left(q;q\right)_{\infty}.
\end{equation}

In the base 3 case, using 
\begin{equation}
\prod_{n\ge1,j\ge0}\left(\frac{a_{\left(3n-1\right).3^{j}}a_{\left(3n-2\right).3^{j}}}{a_{3n.3^{j}}^{p-1}}\right)^{p^{j}}=\prod_{n\ge1}a_{n},
\end{equation}
we have
\begin{equation}
\prod_{n\ge1,j\ge0}\left(1-q^{\left(3n-1\right).3^{j}}\right)=\prod_{n\ge0,j\ge0}\left(1-q^{\left(3n+2\right).3^{j}}\right)=\prod_{j\ge0}\left(q^{2.3^{j}};q^{3^{j+1}}\right)_{\infty}
\end{equation}
and
\begin{equation}
\prod_{n\ge1,j\ge0}\left(1-q^{\left(3n-2\right).3^{j}}\right)=\prod_{n\ge0,j\ge0}\left(1-q^{\left(3n+1\right).3^{j}}\right)=\prod_{j\ge0}\left(q^{3^{j}};q^{3^{j+1}}\right)_{\infty}
\end{equation}
and
\begin{equation}
\prod_{n\ge1,j\ge0}\left(1-q^{3n.3^{j}}\right)=\prod_{n\ge0,j\ge0}\left(1-q^{\left(3n+3\right).3^{j}}\right)=\prod_{j\ge0}\left(q^{3^{j+1}};q^{3^{j+1}}\right)_{\infty}
\end{equation}
so that
\begin{equation}
\prod_{j\ge0}\left(\frac{\left(q^{3^{j}};q^{3^{j+1}}\right)_{\infty}\left(q^{2.3^{j}};q^{3^{j+1}}\right)_{\infty}}{\left(q^{3^{j+1}};q^{3^{j+1}}\right)_{\infty}^{p-1}}\right)^{p^{j}}=\left(q;q\right)_{\infty}.
\end{equation}
\end{proof}

This result can be extended to the more general case of the $q-$Pochhammer
\begin{equation}
\left(a;q\right)_{\infty}=\prod_{n\ge0}\left(1-aq^{n}\right)
\end{equation}
by considering the sequence
\begin{equation}
a_{n}=1-aq^{n-1},\,\,n\ge1.
\end{equation}
\begin{prop}
The following multisection formula holds:
\begin{equation}
\prod_{j\ge0}\left(\frac{\left(aq^{2^{j}-1};q^{2^{j+1}}\right)_{\infty}}{\left(aq^{2^{j+1}-1};q^{2^{j+1}}\right)_{\infty}}\right)^{2^{j}}=\left(a;q\right)_{\infty}.
\end{equation}
\end{prop}
\begin{proof}
We have
\begin{align*}
\prod_{j\ge0,n\ge1}\left(\frac{a_{\left(2n-1\right).2^{j}}}{a_{\left(2n\right).2^{j}}}\right)^{2^{j}} & =\prod_{j\ge0,n\ge1}\left(\frac{1-aq^{\left(2n-1\right).2^{j}-1}}{1-aq^{\left(2n\right).2^{j}-1}}\right)^{2^{j}}\\
 & =\prod_{j\ge0}\left(\frac{\prod_{n\ge0}\left(1-aq^{\left(2n+1\right).2^{j}-1}\right)}{\prod_{n\ge0}\left(1-aq^{\left(2n+2\right).2^{j}-1}\right)}\right)^{2^{j}}\\
 & =\prod_{j\ge0}\left(\frac{\prod_{n\ge0}\left(1-aq^{2^{j}-1}q^{n.2^{j+1}}\right)}{\prod_{n\ge0}\left(1-aq^{2^{j+1}-1}q^{n.2^{j+1}}\right)}\right)^{2^{j}}\\
 & =\prod_{j\ge0}\left(\frac{\left(aq^{2^{j}-1};q^{2^{j+1}}\right)_{\infty}}{\left(aq^{2^{j+1}-1};q^{2^{j+1}}\right)_{\infty}}\right)^{2^{j}}.
\end{align*}
\end{proof}

\begin{rem}
A combinatorial interpretation of identity \eqref{QPochhammer2} is as follows: rewrite \eqref{QPochhammer2} as
\begin{equation}
\label{QPochhammer2inv}    
\frac{1}{\left(q;q\right)_{\infty}}
=\prod_{j\ge0}\frac{1}{\left(q^{2^{j}};q^{2^{j+1}}\right)_{\infty}}.
\end{equation}
The left-hand side is the generating function for the number of partitions $p(n)$ of the integer $n$.
More generally, the infinite product
\begin{equation}
\prod_{i \in I}\frac{1}{1-q^i}
\end{equation}
is the generating function for the number of partitions $p_{I}(n)$ of $n$ with parts belonging to the set $I.$ Hence
\begin{equation}
\frac{1}{\left(q^{2^{j}};q^{2^{j+1}}\right)_{\infty}}
=\prod_{i \in 
I_{j}
}\frac{1}{1-q^i}
\end{equation}
is the generating function for the number of partitions of $n$ with parts in the set $I_{j}=\{2^{j}+\mathbb {N}.2^{j+1},\,\,j\ge0\}.$
It is easy to check that the set of subsets $I_{j},\,\,j\ge 0$ forms a partition of $\mathbb{N}$, i.e.
\begin{equation} \cup_{j\ge 0} I_j = \mathbb{N}\text{ and }
I_j \cap I_k = \varnothing,\,\, j \ne k.
\end{equation}
In fact, the partition $ \cup_{j\ge 0} I_j = \mathbb{N}$ has an elementary interpretation in base 2: since $I_j=\{n:\nu_2(n)=j\}$, this partition decomposes the set of integers $n$ according to their $2-$valuation.

\end{rem}
\section{Applications} \label{sec:Apps}

We begin with a simple application of Example \eqref{eq:q-case}.   
Let

\begin{equation}
 a (m)=1+{x^{2}}/{m^{2}} 
 \label{Am2}
 \end{equation}

so that
\begin{equation}
\prod_{j\ge0,n\ge1}{\left({\displaystyle\frac{a ((2 n -1). 2^{j})}{(a (2 n ).2^{j})^{q-1}}}\right)}^{\displaystyle q^{ j}} =
{ \moverset{\infty}{\munderset{m =1}{\textcolor{gray}{\prod}}}(1+{x^{2}}/{m^{2}})={\sinh (\pi  x)}/{x \pi} }
\label{Ls}
\end{equation}

which, after substitution in the series equivalent case \eqref{q-caseLn}, becomes

\begin{equation}
\sum_{j\ge0,n\ge1}q^{j} \left(\ln ({\displaystyle 1+\frac{x^{2}}{(2 n -1)^{2}2^{2j}})}-(q -1) \ln ({\displaystyle 1+{ \frac{ x^{2}}{ 2^{2j+2}n^{2}}}})\right)= 
{ \ln (\frac{\sinh (\pi  x)}{x \pi}) }\,.
\label{Start}
\end{equation}

\subsection{By differentiation} \label{sec:difEx1}

It is interesting to differentiate \eqref{Start} once with respect to x, and, since there is no overall \textit{q} dependence, letting \textit{q}\textit{=1} solves the double summation 

\begin{equation}
\sum_{j\ge0,n\ge1}{1}/{\left(n (n -1) 2^{2+2 j}+x^{2}+4^{j}\right)}={(\pi  x \coth (\pi  x)-1)}/{2 x^{2}} .
\label{St1a}
\end{equation}

However, since the inner sum is known \cite{Maple23}, \cite{Math23}, that is

\begin{equation}
{ \moverset{\infty}{\munderset{n =1}{\textcolor{gray}{\sum}}}{1}/{(n (n -1) 2^{2+2 j}+x^{2}+4^{j})}={\pi  \tanh ({ {\pi  x}/{2^{(1+j)}}})}/{\left(2^{j +2} x\right)} }
\label{St1b}
\end{equation}

we eventually reproduce a listed (telescoping-based) identity \cite[Eq.  (43.6.4)]{Hansen}

\begin{equation}
   { \moverset{\infty}{\munderset{j =0}{\textcolor{gray}{\sum}}}\tanh ({  x}/{2^{1+j}})/ 2^{j}={2 ( x \coth ( x)-1)}/{  x} }.
   \label{H43}
\end{equation}

Continuing, by differentiating \eqref{Start} with respect to \textit{q}, we obtain 

\begin{equation}
\sum_{j\ge0,n\ge1}j \ln ({1+\displaystyle \frac{x}{2^{2 j}\,(2 n -1)^{2} }})=\sum_{j\ge0,n\ge1}\ln ({\displaystyle 1+\frac{x}{ n^{2}\, 2^{2j+2}}}) 
\label{St3}
\end{equation}

with \textit{q=1.} Again, differentiating \eqref{St3} with respect to x yields the transformation

\begin{equation}
 { \moverset{\infty}{\munderset{j =1}{\textcolor{gray}{\sum}}}\frac{j }{2^{j}}\tanh (\frac{x}{2^{1+j}})=\frac{2}{x} \moverset{\infty}{\munderset{j =0}{\textcolor{gray}{\sum}}}(\frac{x}{2^{1+j}} \coth (\frac{x}{2^{1+j}})-1) }\textit{ .} \label{St4a}  
\end{equation}

\subsection{By power series expansion} \label{sec:PsEx1}
In the following, extensive use is made of the identity \cite{Dieckmann}
\begin{equation}
    \moverset{\infty}{\munderset{k =1}{\prod}}\! \left(1+\left(\frac{x}{k +b}\right)^{n}\right)
 = 
\frac{\Gamma \! \left(1+b \right)^{n} }{b^{n}+x^{n}}\moverset{n}{\munderset{k =1}{\prod}}\! \frac{1}{\Gamma \left(b -x\,{\mathrm e}^{{i\,\pi  \left(2\,k +1\right)}/{n}}\right)}\,,
\label{Dkm}
\end{equation}
valid for $x\in \mathbb{C}$. See also \cite[Eq (27)]{InfProduct}.
Consider the case
\begin{equation}
    a(m)=1-x^k/m^k,\hspace{30pt} m>1
    \label{Ex5b3}
\end{equation}
and recall for example, the simple identity

\begin{equation}
    \ln \! \left(1-\frac{x^{k}}{\left(\left(3 n -2\right) 3^{j}\right)^{k}}\right)
 = 
-\moverset{\infty}{\munderset{p =1}{\textcolor{gray}{\sum}}}\! \frac{1}{p}\left(\frac{x^{k}}{\left(\left(3 n -2\right) 3^{j}\right)^{k}}\right)^{p},\hspace{20pt} |x|<3.
\label{Lex2}
\end{equation}

From \eqref{X5b3}  and \eqref{Lex2} we have, after the (convergent) sums are reordered,
\begin{align}  \label{Ls3C}
     \sum_{j\ge1,n\ge1} &\left(\ln \! \left(a \! \left(\left(3 n -1\right). 3^{j}\right)\right)+\ln \! \left(a \! \left(\left(3 n -2\right). 3^{j}\right)\right)\right)\\ \nonumber
 &= \sum_{p\ge1,j\ge1} \frac{\left({x}/{3^{j}}\right)^{k p}}{p} \moverset{\infty}{\munderset{n =1}{\textcolor{gray}{\sum}}}\! \left(\left(\frac{1}{\left(3 n -1\right)^{k}}\right)^{p}+\left(\frac{1}{\left(3 n -2\right)^{k}}\right)^{p}\right)\,\hspace{10pt}|x|<3.
\end{align}

The sum indexed by $j$ is trivially evaluated, and the sum indexed by $n$ is recognizable as a series representation of a special case of the Hurwitz Zeta function $\zeta \left(s , a\right) \equiv 
\moverset{\infty}{\munderset{n =0}{\sum}}\! {1}/{\left(n +a \right)^{s}}
$, leading to the identity
\begin{equation}
    \moverset{\infty}{\munderset{p =1}{\textcolor{gray}{\sum}}}\! \frac{x^{k p} }{\left(3^{k p}-1\right) p}\left(\zeta \left(k p , \frac{2}{3}\right)+\zeta \left(k p , \frac{1}{3}\right)\right)
 = 
\ln \! \left(-\moverset{\quad\; k -1}{\munderset{\quad\; j =0}{{x^{k}\;\prod}}}\! \Gamma \! \left(-x\,{\mathrm e}^{{2 \,\mathrm{I} \pi  j}/{k}}  \right) \right),\hspace{10pt} |x|<1,
\label{Ls3e2}
\end{equation}

where the right-hand side arises from \eqref{Dkm}, and we have replaced $x:=3x$. Because $k$ and $p$ are both integers and $k>1$, \eqref{Ls3e2} can also be rewritten using \cite[Eq. 25.11.12]{NIST} as
\begin{equation}
\moverset{\infty}{\munderset{p =1}{\sum}}\! \frac{\left(\psi \! \left(k\,p -1, \frac{2}{3}\right)+\psi \! \left(k\,p -1, \frac{1}{3}\right)\right) \left(-x \right)^{k\,p}}{\left(3^{k\,p}-1\right) \Gamma \! \left(k\,p +1\right)}
 = 
\frac{1}{k}\ln \! \left(-x^{k}\moverset{k -1}{\munderset{j =0}{\prod}}\! \Gamma \! \left(-x\;{\mathrm e}^{\frac{2\,i\,\pi \,j}{k}} \right)\right) 
\hspace{20pt} |x|<1,
\label{Ls3e2p}
\end{equation}
where $\psi(n,x)$ is the polygamma function. For the case \eqref{Ls3e2} using k=2, we find 

\begin{equation}
 \moverset{\infty}{\munderset{p =1}{\sum}}\! \frac{x^{2\,p} \left(\psi \! \left(2\,p -1, \frac{2}{3}\right)+\psi \! \left(2\,p -1, \frac{1}{3}\right)\right)}{\left(3^{2\,p}-1\right) \Gamma \! \left(2\,p +1\right)}
 = 
\frac{1}{2}\ln \! \left(\frac{x\,\pi}{\sin \left(\pi \,x \right)}\right)\hspace{20pt} |x|<1,
\label{Lk2}
\end{equation}
for $k=3$ we obtain
\begin{equation}
    \moverset{\infty}{\munderset{p =1}{\sum}}\! \frac{\left(\psi \! \left(3\,p -1, \frac{2}{3}\right)+\psi \! \left(3\,p -1, \frac{1}{3}\right)\right) \left(-1\right)^{p}\,x^{3\,p}}{\left(3^{3\,p}-1\right) \Gamma \! \left(3\,p +1\right)}
 = 
\frac{1}{3}\ln \! \left({x^{2}\,\left| \Gamma \! \left(-\frac{i\,\sqrt{3}\,x}{2}+\frac{x}{2}\right)\right|}^{2}\,\Gamma \! \left(1-x \right)\right)
\label{Lk3}
\end{equation}
and by setting $x:= i\,x$ and $k:=2k$ in \eqref{Ls3e2p} we discover \begin{equation}
    \moverset{\infty}{\munderset{p =1}{\sum}}\! \frac{\left(-1\right)^{p}\,x^{2\,k\,p} \left(\psi \! \left(2\,k\,p -1, \frac{2}{3}\right)+\psi \! \left(2\,k\,p -1, \frac{1}{3}\right)\right)}{\Gamma \! \left(2\,k\,p +1\right) \left(3^{2\,k\,p}-1\right)}
 = 
\frac{\ln \! \left(x^{2\,k}\,\moverset{2\,k -1}{\munderset{j =0}{\prod}}\! \Gamma \! \left(-i\,x\,{\mathrm e}^{\frac{i\,\pi \,j}{k}} \right)\right)}{2\,k}\,
\label{Leven}
\end{equation}
corresponding to the related case
\begin{equation}
    a(m)=1+x^k/m^k,\hspace{30pt} m>1\,.
    \label{Ex5b4}
\end{equation}
Similarly let $x:=ix$ and $k:=2k+1$, to find
\begin{align}  \label{Lodd}
    \moverset{\infty}{\munderset{p =1}{\sum}}\! \frac{\left(\psi \! \left(\left(2\,k +1\right) p -1, \frac{2}{3}\right)+\psi \! \left(\left(2\,k +1\right) p -1, \frac{1}{3}\right)\right) x^{\left(2\,k +1\right) p}\,{\mathrm e}^{-\frac{i\,\pi \,p}{2}} \left(-1\right)^{k\,p}}{\left(3^{\left(2\,k +1\right) p}-1\right) \Gamma \! \left(\left(2\,k +1\right) p +1\right)} \\\nonumber
 = 
\frac{\ln \! \left(-i\,\left(-1\right)^{k}\,x^{2\,k +1} \moverset{2\,k}{\munderset{j =0}{\prod}}\! \Gamma \! \left(-i\,x\,{\mathrm e}^{\frac{2\,i\,\pi \,j}{2\,k +1}} \right)\right) }{2\,k +1}
\end{align}

Choosing $k=1$ in \eqref{Leven}, or $k=0$ in the real part of \eqref{Lodd}, gives
\begin{equation}
   \moverset{\infty}{\munderset{p =1}{\sum}}\! \frac{\left(-1\right)^{p}\,x^{2\,p} \left(\psi \! \left(2\,p -1, \frac{2}{3}\right)+\psi \! \left(2\,p -1, \frac{1}{3}\right)\right)}{\left(3^{2\,p}-1\right) \Gamma \! \left(2\,p +1\right)}
 = 
\frac{1}{2}\ln \! \left(\frac{x\,\pi}{\sinh \left(\pi \,x \right)}\right)\,.
\label{Lodd2aR}
\end{equation}
\subsection{}
We continue with another simple example that utilizes the methods applied above. Let \begin{equation}
    a(m)=1/(1+x^k/m^k).
    \label{AmEx1}
\end{equation}
Then, according to \eqref{eq:general} 
\begin{equation}
    \moverset{\infty}{\munderset{j =0}{\prod}}\! {\left(\moverset{\infty}{\munderset{n =1}{\prod}}\! {\left(1+\frac{x^{s}}{\left(2^{j +1}\right)^{s}\,n^{s}}\right)}/{\left(1+\frac{x^{s}}{\left(2^{j +1}\right)^{s} \left(n -\frac{1}{2}\right)^{s}}\right)}\right)}^{2^{j}}
 = 
\moverset{\infty}{\munderset{m =1}{\prod}}\! \frac{1}{1+\frac{x^{s}}{m^{s}}}\,,
\label{Ls1Rs1}
\end{equation}
a result that can be tested numerically for $s\geq 2$. In the case that $s=n$ using \eqref{Dkm} we find the general identity
\begin{align} \label{Ls2}
    \moverset{\infty}{\munderset{j =0}{\prod}}\! {\left(\frac{\left(1+\left(-\frac{2^{j}}{x}\right)^{n}\right) }{\pi^{\frac{n}{2}}}\moverset{n}{\munderset{k =1}{\prod}}\! \frac{\Gamma \left(-\frac{1}{2}-\frac{x}{2^{j +1}}\,{\mathrm e}^{{i\,\pi  \left(2\,k +1\right)}/{n}}\right)}{\Gamma \left(-\frac{x}{2^{j +1}}\,{\mathrm e}^{{i\,\pi  \left(2\,k +1\right)}/{n}}\right)}\right)}^{2^{j}}
 = \moverset{\infty}{\munderset{m =1}{\prod}}\! \frac{1}{1+{x^{n}}/{m^{n}}}
\end{align}
If $n=2$ in \eqref{Ls2}, after some simplification involving \cite[Eqs. (5.4.3) and (5.4.4)] {NIST}, we arrive at the original curious and provocative result \eqref{eq:tan}. In the case that $n=4$ and $x\in \mathbb{C}$, after further simplification and the redefinition $x:=2\sqrt{2}\,x/\pi$, we find
\begin{equation}
\moverset{\infty}{\munderset{j =1}{\prod}}\! {\left(\frac{ 2^{2\,j}}{2\,x^{2}}\tan \! \left(\frac{\left(1+i\right) x}{2^{j}}\right) \tan \! \left(\frac{\left(1-i\right) x}{2^{j}}\right)\right)}^{2^{j}}
 = 
\frac{16\,x^{4}}{{\left(2 \left(\cosh^{2}\left(x \right)\right)-1-\cos \! \left(2\,x \right)\right)}^{2}}
\,.
\label{Ls5d}
\end{equation}
\subsection{A sum involving $\zeta(k)$}
With the slight variation $a(m)=1+x^k/m^k$ applied to \eqref{q-caseLn} using $q=1$ and analyzed as above, we find the identities
\begin{equation}
   \moverset{\infty}{\munderset{j =1}{\sum}}\! \frac{\left(-1\right)^{j +1}\,x^{2\,j\,k}\,\zeta \! \left(2\,j\,k \right)}{j}
 = 
\ln \! \left(\frac{\left(-i\right)^{k}}{\pi^{k}\,x^{k}} \moverset{k}{\munderset{j =1}{\prod}}\! \sin \! \left(\left(-1\right)^{\frac{2\,j -1}{2\,k}}\,\pi \,x \right)\right)
\label{G1}
\end{equation}
and
\begin{equation}
    \moverset{\infty}{\munderset{j =1}{\sum}}\! \frac{\left(-1\right)^{j +1}\,x^{\left(2\,k +1\right)\, j}\,\zeta \! \left(\left(2\,k +1\right) j \right)}{j}
 = 
\ln \! \left(\frac{1}{\Gamma \! \left(x +1\right)}\moverset{2\,k}{\munderset{j =1}{\prod}}\! \frac{1}{\pi}\,\sin \left(\pi  \left(-1\right)^{\frac{j \left(2+2\,k \right)}{2\,k +1}}\,x \right) \Gamma \left(\left(-1\right)^{\frac{j \left(2+2\,k \right)}{2\,k +1}+1}\,x \right)\right)
\label{G2}
\end{equation}
where the right-hand sides of both are available from \eqref{Dkm}. In the case $k=1$, \eqref{G2} reduces to
\begin{equation}
    \moverset{\infty}{\munderset{k =1}{\sum}}\! \frac{x^{k}\,\zeta \! \left(3\,k \right)}{k}
 = 
\ln \! \left(\Gamma \! \left(1-x^{\frac{1}{3}}\right) \Gamma \! \left(1+\frac{x^{\frac{1}{3}} \left(1+i\,\sqrt{3}\right)}{2}\right) \Gamma \! \left(1+\frac{x^{\frac{1}{3}} \left(1-i\,\sqrt{3}\right)}{2}\right)\right) \hspace{20pt} x<1,
\label{St3x}
\end{equation}
an identity known to Mathematica. See also \cite[section 3.4]{Sriv&Choi}
\subsection{An infinite product study}
Consider the case
\begin{equation}
    a \! \left(m \right) = 1+{\mathrm e}^{-\pi  \left(2\,m +1\right)}
    \label{H89}
\end{equation}
given that \cite[Eq. (89.21.4)]{Hansen}
\begin{equation}
    \moverset{\infty}{\munderset{m =1}{\prod}}\! \left(1+{\mathrm e}^{-\pi  \left(2\,m +1\right)}\right)
 = 
\frac{2^{\frac{1}{4}}\,{\mathrm e}^{-\frac{\pi}{24}}}{1+{\mathrm e}^{-\pi}}\,.
\label{H1a}
\end{equation}
Apply \eqref{H89} to \eqref{q-caseLn} with $q=1$ (see Example \ref{eq:X4}) to obtain the additive multisection
\begin{equation}
    \sum_{j\ge0,k\ge1} \frac{{\mathrm e}^{-k\,\pi} \left(-1\right)^{k}}{k\,\sinh \! \left(2\,k\,\pi \,2^{j}\right)}
 = 
-2\,\ln \! \left(\frac{2^{\frac{1}{4}}\,{\mathrm e}^{-{\pi}/{24}}}{1+{\mathrm e}^{-\pi}}\right)\,
\label{Ls3B}
\end{equation}
after expanding the logarithmic terms in analogy to \eqref{Lex2}. The outer sum (over $j$) can be evaluated by applying the listed identity \cite[Eq. (25.1.1)]{Hansen} after setting $x:=ix$ and evaluating the limit $n\rightarrow \infty$ in that identity, to yield
\begin{equation}
    \moverset{\infty}{\munderset{j =0}{\sum}}\! \frac{1}{\mathrm{sinh}\! \left(2\,k\,\pi \,2^{j}\right)}
 = 
\coth \! \left(2\,k\,\pi \right)-1+\mathrm{csch}\! \left(2\,k\,\pi \right)\,,
\label{H25p1p1}
\end{equation}
in which case \eqref{Ls3B} reduces, after some simplification, to
\begin{equation}
    \moverset{\infty}{\munderset{k =1}{\sum}}\! \frac{\left(-1\right)^{k}\,{\mathrm e}^{-k\,\pi} \coth \! \left(k\,\pi \right)}{k}
 = 
\ln \! \left(\frac{\left(1+{\mathrm e}^{-\pi}\right)\, {\mathrm e}^{{\pi}/{12}}}{\sqrt{2}}\right)\,.
\label{Ls3D}
\end{equation}
Comparing \eqref{Ls3D} with a naive expansion of $\ln \left(1+{\mathrm e}^{-\pi  \left(2\,m +1\right)}\right)$ in \eqref{H1a}, verifies the simple transformation
\begin{equation}
    \moverset{\infty}{\munderset{k =1}{\sum}}\! \frac{\left(-1\right)^{k}\,{\mathrm e}^{-2\,k\,\pi}}{k\,\sinh \! \left(k\,\pi \right) }
 = 
\moverset{\infty}{\munderset{k =1}{\sum}}\! \frac{\left(-1\right)^{k}\,{\mathrm e}^{-k\,\pi} \coth \! \left(k\,\pi \right)}{k}+\ln \! \left(1+{\mathrm e}^{-\pi}\right),
\label{Ls3E}
\end{equation}
so that, comparing \eqref{Ls3D} and \eqref{Ls3E} identifies
\begin{equation}
    \moverset{\infty}{\munderset{k =1}{\sum}}\! \frac{\left(-1\right)^{k}\,{\mathrm e}^{-2\,k\,\pi}}{k\,\sinh \! \left(k\,\pi \right) }
 = 
\ln \! \left(\frac{\left(1+{\mathrm e}^{-\pi}\right)^{2}\,{\mathrm e}^{{\pi}/{12}}}{\sqrt{2}}\right)\,.
\label{Ls3F}
\end{equation}

Similarly, letting $q=2$ in \eqref{q-caseLn} leads to

\begin{equation}
    \moverset{\infty}{\munderset{j =0}{\sum}}\! 2^{j} \moverset{\infty}{\munderset{k =1}{\sum}}\! \frac{\left(-1\right)^{k}}{k} \moverset{\infty}{\munderset{n =1}{\sum}}\! \left(-{\mathrm e}^{-k \left(2 \left(2\,n -1\right) 2^{j}+1\right) \pi}+{\mathrm e}^{-k \left(4\,n\,2^{j}+1\right) \pi}\right)
 = 
\ln \! \left(\frac{2^{\frac{1}{4}}\,{\mathrm e}^{-\frac{\pi}{24}}}{1+{\mathrm e}^{-\pi}}\right)\,.
\label{H1Lb}
\end{equation}
The innermost sum (over $n$) can be written in terms of hyperbolic functions, eventually producing the double sum identity
\begin{equation}
   \moverset{\infty}{\munderset{k =1}{\sum}}\! \frac{\left(-1\right)^{k} }{k}\moverset{\infty}{\munderset{j =0}{\sum}}\! \frac{2^{j}\,{\mathrm e}^{-k\,\pi  \left(2^{j}+1\right)}}{\cosh \left(k\,\pi \,2^{j}\right)}
 = 
\ln \! \left(\frac{\left(1+{\mathrm e}^{-\pi}\right)^{2}\,{\mathrm e}^{{\pi}/{12}}}{\sqrt{2}}\right)\,
\label{H1Ld}
\end{equation}
after transposing the two sums. Noting the equality of the right-hand sides of \eqref{Ls3F} and \eqref{H1Ld}, suggests that the inner sum of \eqref{H1Ld} (over $j$) equates to the equivalent terms in the summand of \eqref{Ls3F}, i.e.
\begin{equation}
   \moverset{\infty}{\munderset{j =0}{\sum}}\frac{2^{j}\,{\mathrm e}^{-k\,\pi  \left(2^{j}+1\right)}}{\cosh \! \left(k\,\pi \,2^{j}\right)}
 = \frac{{\mathrm e}^{-2\,k\,\pi}}{\sinh \! \left(k\,\pi \right)}\,,
\label{Cj}
\end{equation}
a proof of which can be found in Appendix \ref{sec:Proof8p34}.
\subsection{Application of (\ref{eq:1/(k+1)})}

Consider the case \eqref{eq:1/(k+1)} using $a_{m}=1+x^{2}/m^{2}$, giving
\begin{equation}
    \moverset{\infty}{\munderset{j=0\,n =1}{\prod}}\! \left(1+\frac{1}{\left(2\,n -1\right)^{2}\,2^{2\,j}}\right)^{\frac{1}{j +1}} \left(1+\frac{1}{4\,n^{2}\,2^{2\,j}}\right)^{\frac{1}{\left(j +1\right) \left(j +2\right)}}
 = \frac{\sinh \! \left(\pi \right)}{\pi}\,.
\label{T2}
\end{equation}
From \eqref{Dkm} we have
\begin{equation}
    \moverset{\infty}{\munderset{n =1}{\prod}}\! \left(1+\frac{1}{\left(2\,n -1\right)^{2} \left(2^{j}\right)^{2}}\right)
 = \cosh \! \left(\frac{\pi}{2^{j +1}}\right)
\label{Math1}
\end{equation}
and
\begin{equation}
    \moverset{\infty}{\munderset{n =1}{\prod}}\! \left(1+\frac{1}{4\,n^{2} \left(2^{j}\right)^{2}}\right)
 = \frac{ 2^{j +1}}{\pi}\sinh \! \left(\frac{\pi}{2^{j +1}}\right)
\label{Math1a}
\end{equation}
leading to the identity
\begin{equation}
    \moverset{\infty}{\munderset{j =0}{\prod}}\! \left(\cosh^{\frac{1}{j +1}}\left({\pi \,x}/{2^{j +1}}\right)\right) {\left(\frac{\sinh \! \left({\pi \,x}/{2^{j +1}}\right) 2^{j +1}}{\pi \,x}\right)}^{\frac{1}{\left(j +1\right) \left(j +2\right)}}
 = \frac{\sinh \! \left(\pi \,x \right)}{\pi \,x}
\label{T2a}
\end{equation}
or its equivalent
\begin{equation}
    \moverset{\infty}{\munderset{j =1}{\sum}}\! \frac{\ln \! \left(\cosh \! \left({\pi \,x}/{2^{j +1}}\right)\right)+\ln \left(\frac{ 2^{j +1}}{\pi \,x}\sinh \left({\pi \,x}/{2^{j +1}}\right)\right)/(j+2)}{j +1}
 = 
\frac{1}{2}\ln \! \left(\frac{2\,}{\pi \,x}\sinh \left(\frac{\pi \,x}{2}\right)\right)\,.
\label{T2B}
\end{equation}
Differentiating \eqref{T2B} with respect to $x$ yields
\begin{equation}
    \moverset{\infty}{\munderset{j =1}{\sum}}\! \frac{2^{-j -1} }{j +1}\left(\tanh \! \left(\pi \,x\,2^{-j -1}\right)+\frac{1}{j +2}\coth \left(\pi \,x\,2^{-j -1}\right)\right)
 = \frac{\coth \! \left(\frac{\pi \,x}{2}\right)}{4}
\label{T2c}
\end{equation}
and further differentiating produces
\begin{equation}
    \moverset{\infty}{\munderset{j =1}{\sum}}\! \frac{2^{-2\,j -2} }{j +1}\left(\mathrm{sech}^{2} \left(\pi \,x\,2^{-j -1}\right)-\frac{1}{j +2}\mathrm{csch}^{2}\left(\pi \,x\,2^{-j -1}\right)\right)
 = -\frac{\mathrm{csch}\! \left(\frac{\pi \,x}{2}\right)^{2}}{8}\,.
\label{T2d}
\end{equation}
In \eqref{T2a} setting $x:=ix$ gives
\begin{equation}
    \moverset{\infty}{\munderset{j =1}{\prod}}\! {\left(2\,\cos \! \left(\pi \,x\,2^{-j -1}\right)\right)}^{\frac{1}{j +1}}\,{\left(\frac{\sin \! \left(\pi \,x\,2^{-j -1}\right)}{2\,\pi \,x}\right)}^{\frac{1}{\left(j +1\right) \left(j +2\right)}}
 = \sqrt{\frac{2\,}{\pi \,x}\sin \! \left(\frac{\pi \,x}{2}\right)}\,,\hspace{20pt} |x|<4.
\label{T2x}
\end{equation}
Further, employing $a_{m}=1+x^3/m^3$ with $ x\in \Re$, produces the identity
\begin{align} \label{t2a}
    \moverset{\infty}{\munderset{j =0}{\prod}}\! \left(\frac{\pi^{\frac{3}{2}}}{\Gamma \! \left(\frac{1}{2}+2^{-j -1}\,x\right) {| \Gamma \! \left(\frac{1}{2}-2^{-j -2}\,x \left(1-i\,\sqrt{3}\right)\right)|}^{2}}\right)^{\frac{1}{j +1}}\,&{\left(\frac{1}{\Gamma \! \left(1+2^{-j -1}\,x\right) {| \Gamma \! \left(1-2^{-j -2}\,x \left(1-i\,\sqrt{3}\right)\right)|}^{2}}\right)}^{\frac{1}{\left(j +1\right) \left(j +2\right)}} \\ \nonumber
& = 
\frac{1}{\Gamma \! \left(x +1\right) {| \Gamma \! \left(1-\frac{x }{2}\left(1-i\,\sqrt{3}\right)\right)|}^{2}}
\hspace{20pt}\,.
\end{align}
\subsection{A recurring family}
Consider the identity \eqref{Bcase} using $b(m)=\exp(-zm),~Re(z)>0$. The sums over $n$ on both sides can be evaluated in closed form, leading to the transformation
\begin{equation}
    \moverset{\infty}{\munderset{j =1}{\sum}}\! \frac{B_{2\,p +1}\! \left(j \right) {\mathrm e}^{z2^{j}}}{\left({\mathrm e}^{z\,2^{j +1}}-1\right) \left(2\,p +1\right)}
 = 
\moverset{\infty}{\munderset{j =1}{\sum}}\! \frac{j^{2\,p}}{{\mathrm e}^{z\,2^{j +1}}-1}\,,
\label{Eq3p3d}
\end{equation}
since $B_{2p+1}(0)=0$. We notice that, with $f(j)={B_{2p+1}(j)}/(2p+1)$ and $g(j)={1}/{\left(1-e^{z\,2^j}\right)}$ and with the finite difference operator $\Delta f(j) = f(j+1)-f(j),$ we have
\begin{equation}
\Delta f(j) = j^{2p},\,\,
\Delta g(j)= 
\frac{e^{z\,2^j}}{e^{z\,2^{j+1}}-1}
\end{equation}
so that \eqref{Eq3p3d} can be interpreted as a summation by parts identity 
\begin{equation}
\sum_{j\ge1}f(j)\Delta g(j) = - \sum_{j\ge1}g(j+1) \Delta  f(j). 
\end{equation}
\\
More interesting is the case $b(m)=\exp(-x\,m^2),~x>0$ applied to \eqref{Bcase}, leading to
\begin{equation}
    \moverset{\infty}{\munderset{j =1}{\sum}}\! \frac{B_{2\,p +1}\! \left(j \right)}{2\,p +1} \vartheta_{2}\! \left(0, {\mathrm e}^{-2^{2+2\,j}}\right)
 = 
\moverset{\infty}{\munderset{j =1}{\sum}}\! j^{2\,p} \left(\vartheta_{3}\! \left(0, {\mathrm e}^{-2^{2+2\,j}}\right)-1\right)
\label{Eq8A}
\end{equation}
by recognizing that
\begin{equation}
    \moverset{\infty}{\munderset{n =1}{\sum}}\! {\mathrm e}^{-4\,x\,n^{2} 2^{2\,j}}
 = 
\frac{1}{2}\left(\vartheta_{3}\! \left(0, {\mathrm e}^{-x\,2^{2+2\,j}}\right)-1\right)
\label{S1}
\end{equation}
and
\begin{equation}
    \moverset{\infty}{\munderset{n =1}{\sum}}\! {\mathrm e}^{-x\,\left(2\,n -1\right)^{2}\,4^{j}}
 = \frac{1}{2}\vartheta_{2}\! \left(0, {\mathrm e}^{-x\,2^{2+2\,j}}\right)\,
\label{Math2}
\end{equation}
coincide with the usual \cite[section 20.2]{NIST} definitions of the Jacobi theta functions $\vartheta_{2}$ and $\vartheta_{3}$. Continuing with the choice $b(m)=\exp(-x\,m^2)$, the two identities \eqref{S1} and \eqref{Math2} recur in many other identities presented in previous sections. Applied to \eqref{eq:1/(k+1)B} we find
\begin{equation}
    \moverset{\infty}{\munderset{j =0}{\sum}}\! \left(\frac{\vartheta_{2}\! \left(0, {\mathrm e}^{-x\,2^{2+2\,j}}\right)}{j +1}+\frac{\vartheta_{3}\! \left(0, {\mathrm e}^{-x\,2^{2+2\,j}}\right)-1}{\left(j +1\right) \left(j +2\right)}\right)
 = \vartheta_{3}\! \left(0, {\mathrm e}^{-x}\right)-1\,,
\label{Eq8B1}
\end{equation}
and applied to \eqref{Ex312a} we obtain
\begin{equation}
    \moverset{\infty}{\munderset{j =0}{\sum}}\! \left(\left(-j^{2}+j +1\right)\,\vartheta_{2}\! \left(0, {\mathrm e}^{-x\,2^{2+2\,j}}\right) +2\,j \left(\vartheta_{3}\! \left(0, {\mathrm e}^{-x\,2^{2+2\,j}}\right)-1\right)\right)
 = \vartheta_{3}\! \left(0, {\mathrm e}^{-x}\right)-1\,.
\label{Eq8D1}
\end{equation}
\section{additional identities}
We list in this section two additional identities as a consequence of the dissection identity
\begin{equation}
\prod_{m\ge1}a_{m}=\prod_{j\ge0,n\ge1}a_{\left(2n-1\right).2^{j}}\,.
\label{DissectId}
\end{equation}

\begin{prop}
The function
\begin{equation}
f\left(a,z\right)=\frac{\Gamma^{2}\left(a+1\right)}{\Gamma\left(a+1-\imath z\right)\Gamma\left(a+1+\imath z\right)}
\end{equation}
satisfies the identity
\begin{equation}
f\left(a,z\right)=\prod_{j\ge0}f\left(\frac{a}{2^{j+1}}-\frac{1}{2},\frac{z}{2^{j+1}}\right).
\end{equation}
\begin{proof}
The function $f\left(a,z\right)$ has the infinite product representation
\begin{equation}
f\left(a,z\right)=\prod_{m\ge1}\left(1+\left(\frac{z}{m+a}\right)^{2}\right).
\end{equation}
We deduce from the identity \eqref{DissectId}
that
\begin{align*}
f\left(a,z\right) & =\prod_{j\ge0,n\ge1}\left(1+\left(\frac{z}{\left(2n-1\right).2^{j}+a}\right)^{2}\right)\\
 & =\prod_{j\ge0,n\ge1}\left(1+\left(\frac{z}{n.2^{j+1}+a-2^{j}}\right)^{2}\right)\\
 & =\prod_{j\ge0,n\ge1}\left(1+\left(\frac{\frac{z}{2^{j+1}}}{n+\left(\frac{a}{2^{j+1}}-\frac{1}{2}\right)}\right)^{2}\right)\\
 & =\prod_{j\ge0}f\left(\frac{a}{2^{j+1}}-\frac{1}{2},\frac{z}{2^{j+1}}\right).
\end{align*}
\end{proof}
Explicitly,

\begin{equation}
\frac{\Gamma^{2}\left(a+1\right)}{\Gamma\left(a+1-\imath z\right)\Gamma\left(a+1+\imath z\right)}=\prod_{j\ge0}\frac{\Gamma^{2}\left(\frac{a}{2^{j+1}}+\frac{1}{2}\right)}{\Gamma\left(\frac{a}{2^{j+1}}+\frac{1}{2}-\imath\frac{z}{2^{j+1}}\right)\Gamma\left(\frac{a}{2^{j+1}}+\frac{1}{2}+\imath\frac{z}{2^{j+1}}\right)}.
\label{P9p1}
\end{equation}
If $z=x,~x \in \Re$, because $\Gamma(x)$ is its own complex conjugate, \eqref{P9p1} can be rewritten
\begin{equation}
\frac{\Gamma^{2}\left(a+1\right)}{\left|\Gamma\left(a+1+\imath x\right)\right|^{2}}=\prod_{j\ge0}\frac{\Gamma^{2}\left(\frac{a}{2^{j+1}}+\frac{1}{2}\right)}{\left|\Gamma\left(\frac{a}{2^{j+1}}+\frac{1}{2}+\imath\frac{x}{2^{j+1}}\right)\right|^{2}}.
\label{P9p1a}
\end{equation}
With $a=0,$ we have 
\begin{equation}
f\left(0,z\right)=\frac{\sinh\left(\pi z\right)}{\pi z},\,\,f\left(-\frac{1}{2},z\right)=\frac{1}{\cosh\left(\pi z\right)}
\end{equation}
and we deduce
\begin{equation}
\frac{\sinh\left(\pi z\right)}{\pi z}=\prod_{j\ge0}{\cosh\left(\frac{\pi z}{2^{j+1}}\right)},
\end{equation}
reproducing the listed identity \cite[Eq. (92.1.3)]{Hansen}. \newline 

When written in sum form, \eqref{P9p1} becomes
\begin{align}\label{P9p1b}
\moverset{\infty}{\munderset{j =0}{\sum}}\! \left(2\,\ln \! \left(\Gamma \! \left(\frac{a}{2^{j +1}}+\frac{1}{2}\right)\right)-\right. & \left. \ln \! \left(\Gamma \! \left(\frac{a}{2^{j +1}}+\frac{1}{2}-\frac{i\,z}{2^{j +1}}\right)\right)-\ln \! \left(\Gamma \! \left(\frac{a}{2^{j +1}}+\frac{1}{2}+\frac{i\,z}{2^{j +1}}\right)\right)\right) \\ \nonumber
 &= 
\ln \! \left(\frac{\Gamma \! \left(a +1\right)^{2}}{\Gamma \! \left(-i\,z +a +1\right) \Gamma \! \left(i\,z +a +1\right)}\right)\,,
\end{align}
so that, by either differentiating with respect to $a$ or $z$, and demanding that $z:=x \in \Re $, we respectively find two Euler sums:
\begin{equation}
    \moverset{\infty}{\munderset{j =0}{\sum}}\; 2^{-j}\left(\psi \! \left(a\,2^{-j -1}+\frac{1}{2}\right)-\Re \! \left(\psi \! \left(\frac{i\,2^{-j}\,x}{2}+\frac{1}{2}+a\,2^{-j -1}\right)\right)\right) 
 = 
2\,\psi \! \left(a +1\right)-2\,\Re \! \left(\psi \! \left(i\,x +a +1\right)\right)
\label{P9a}
\end{equation}
\end{prop}
and
\begin{equation}
    \moverset{\infty}{\munderset{j =0}{\sum}}\! 2^{-j}\,\Im \! \left(\psi \! \left(\frac{i\,2^{-j}\,x}{2}+\frac{1}{2}+a\,2^{-j -1}\right)\right)
 = 2\,\Im \! \left(\psi \! \left(i\,x +a +1\right)\right)\,.
\label{P9b}
\end{equation}
\begin{prop}
The function
\begin{equation}
g\left(a,b,z\right)=\frac{b}{\left(b-z\right)}\frac{\Gamma\left(a-\sqrt{a^{2}-b}\right)\Gamma\left(a+\sqrt{a^{2}-b}\right)}{\Gamma\left(a-\sqrt{a^{2}-b+z}\right)\Gamma\left(a+\sqrt{a^{2}-b+z}\right)}
\end{equation}
satisfies the identity
\begin{equation}
g\left(a,b,z\right)=\prod_{j\ge0}g\left(\frac{a}{2^{j}}-1,1-\frac{a}{2^{j+1}}+\frac{b}{2^{2j+2}},\frac{z}{2^{2j+2}}\right).
\end{equation}

\end{prop}
\begin{proof}
The function $g\left(a,b,z\right)$ has infinite product representation
\begin{align*}
g\left(a,b,z\right) & =\prod_{m\ge1}\left(1-\frac{z}{m^{2}+2am+b}\right).
\end{align*}
We deduce
\begin{align*}
g\left(a,b,z\right) & =\prod_{n\ge1,j\ge0}\left(1-\frac{z}{\left(2n-1\right)^{2}2^{2j}+2a\left(2n-1\right)2^{j}+b}\right)\\
 & =\prod_{n\ge1,j\ge0}\left(1-\frac{z}{2^{2j+2}n^{2}-2^{2j+2}n+2^{2j}+2^{j+2}an+b-2^{j+1}a}\right)\\
 & =\prod_{n\ge1,j\ge0}\left(1-\frac{z}{2^{2j+2}n^{2}+n\left(-2^{2j+2}+2^{j+2}a\right)+\left(2^{2j}+b-2^{j+1}a\right)}\right)\\
 & =\prod_{n\ge1,j\ge0}\left(1-\frac{\frac{z}{2^{2j+2}}}{n^{2}+n\left(\frac{a}{2^{j}}-1\right)+\left(1+\frac{b}{2^{2j+2}}-\frac{a}{2^{j+1}}\right)}\right)\\
 & =\prod_{j\ge0}g\left(\frac{a}{2^{j}}-1,1-\frac{a}{2^{j+1}}+\frac{b}{2^{2j+2}},\frac{z}{2^{2j+2}}\right).
\end{align*}

\end{proof}
Explicitly,
\begin{align*}
&\frac{b}{\left(b-z\right)}\frac{\Gamma\left(a-\sqrt{a^{2}-b}\right)\Gamma\left(a+\sqrt{a^{2}-b}\right)}{\Gamma\left(a-\sqrt{a^{2}-b+z}\right)\Gamma\left(a+\sqrt{a^{2}-b+z}\right)}\\
&=\prod_{j\ge0}\frac{1-\frac{a}{2^{j+1}}+\frac{b}{2^{2j+2}}}{\left(1-\frac{a}{2^{j+1}}+\frac{b-z}{2^{2j+2}}\right)}\frac{\Gamma\left(\frac{a}{2^{j}}-1-\sqrt{\left(\frac{a}{2^{j}}-1\right)^{2}-\left(1-\frac{a}{2^{j+1}}+\frac{b}{2^{2j+2}}\right)}\right)}{\Gamma\left(\frac{a}{2^{j}}-1-\sqrt{\left(\frac{a}{2^{j}}-1\right)^{2}-\left(1-\frac{a}{2^{j+1}}+\frac{b}{2^{2j+2}}\right)+\frac{z}{2^{2j+2}}}\right)}\\
&\times\prod_{j\ge0}\frac{\Gamma\left(\frac{a}{2^{j}}-1+\sqrt{\left(\frac{a}{2^{j}}-1\right)^{2}-\left(1-\frac{a}{2^{j+1}}+\frac{b}{2^{2j+2}}\right)}\right)}{\Gamma\left(\frac{a}{2^{j}}-1+\sqrt{\left(\frac{a}{2^{j}}-1\right)^{2}-\left(1-\frac{a}{2^{j+1}}+\frac{b}{2^{2j+2}}\right)+\frac{z}{2^{2j+2}}}\right)}
\end{align*}
\begin{appendix} \label{sec:AppA}
\section{Proof of identity \eqref{Cj}} \label{sec:Proof8p34}
Identity \eqref{Cj} claims that
\begin{equation}
\sum_{j\ge0}2^{j}\frac{e^{-k\pi\left(2^{j}+1\right)}}{\cosh\left(k\pi2^{j}\right)}=\frac{e^{-2k\pi}}{\sinh\left(k\pi\right)}.
\end{equation}
Denoting $q=e^{-k\pi},$ this is
\begin{equation}
\sum_{j\ge0}2^{j+1}\frac{q^{2^{j}+1}}{q^{2^{j}}+q^{-2^{j}}}=\frac{2q^{2}}{q^{-1}-q}
\end{equation}
or
\begin{equation}
\sum_{j\ge0}2^{j+1}\frac{1}{q^{-2^{j+1}}+1}=\frac{2q^{2}}{1-q^{2}}.
\end{equation}
Substituting $q\mapsto\frac{1}{q}$ produces
\begin{equation}
\sum_{j\ge0}2^{j+1}\frac{1}{q^{2^{j+1}}+1}=\frac{2}{q^{2}-1}.
\label{AltEq}    
\end{equation}

We know from \eqref{eq:Teixeira} that
\begin{equation}
\frac{q}{1-q}=\sum_{j\ge0}2^{j}\frac{q^{2^{j}}}{q^{2^{j}}+1}.
\end{equation}
Substituting $q\mapsto\frac{1}{q}$ produces
\begin{equation}
\frac{1}{q-1}=\sum_{j\ge0}2^{j}\frac{1}{q^{2^{j}}+1}
\end{equation}
so that
\begin{equation}
\sum_{j\ge0}2^{j+1}\frac{1}{q^{2^{j+1}}+1}=\frac{1}{q-1}-\frac{1}{1+q}=\frac{2}{q^{2}-1},
\end{equation}
which is \eqref{AltEq}.
\end{appendix}
\bibliographystyle{unsrtnat}

\end{document}